\documentclass{article}

\usepackage{amsmath}
\usepackage{amsthm}
\usepackage{amsfonts}
\usepackage{amssymb}
\usepackage{subfigure}
\usepackage{lpic}
\usepackage{url}
\usepackage{algorithm,algorithmic}
\usepackage{longtable}
\usepackage[margin=1in]{geometry}
\usepackage{fancyhdr}
\usepackage{hyperref}
\usepackage{graphicx}
\usepackage{url}
\usepackage{verbatim}
\usepackage{enumerate}
\usepackage{color}
\definecolor{red}{rgb}{1,0,0}

\definecolor{blu}{rgb}{0,0,1}

\theoremstyle{plain}

\newtheorem{thm}{Theorem}[section]
\newtheorem{cor}[thm]{Corollary}
\newtheorem{lem}[thm]{Lemma}
\newtheorem{prop}[thm]{Proposition}
\newtheorem{conj}[thm]{Conjecture}
\newtheorem{obs}[thm]{Observation}

\newtheorem{quest}[thm]{Question}

\theoremstyle{definition}
\newtheorem{rem}[thm]{Remark}

\theoremstyle{definition}
\newtheorem{defn}[thm]{Definition}

\theoremstyle{definition}
\newtheorem{ex}[thm]{Example}

\newcommand{\R}{\mathbb{R}}
\newcommand{\Z}{\mathbb{Z}}

\newcommand{\bit}{\begin{itemize}}
\newcommand{\eit}{\end{itemize}}
\newcommand{\ben}{\begin{enumerate}}
\newcommand{\een}{\end{enumerate}}
\newcommand{\beq}{\begin{equation}}
\newcommand{\eeq}{\end{equation}}
\newcommand{\bpf}{\begin{proof}}
\newcommand{\epf}{\end{proof}\ms}
\newcommand{\ms}{\medskip}

\newcommand{\dsp}{\displaystyle}

\newcommand{\aw}{\operatorname{aw}}
\newcommand{\sz}{\operatorname{sz}}

\newcounter{casenum}	
\newcounter{subcasenum}	
\numberwithin{subcasenum}{casenum}
\newcounter{subsubcasenum}	
\numberwithin{subsubcasenum}{subcasenum}

\renewcommand{\thecasenum}{\arabic{casenum}}

\newcounter{stagenum}

\newenvironment{mycases}
{
  \list{}{%
    \leftmargin0.5cm   % this is the adjusting screw
    \rightmargin0cm%\leftmargin
  }
  \item\relax
	\setcounter{casenum}{0}
}
{	
	\endlist
}

\newcommand{\mycase}[1]{
	\vspace{0.5em}
	
	\refstepcounter{casenum}
	\noindent\hspace{-0.5cm}\textit{Case \thecasenum: #1}
}

\title{Rainbow arithmetic progressions}
\author{Steve Butler$^{1}$
\and Craig Erickson$^{2}$
\and Leslie Hogben$^{1,3}$
\and Kirsten Hogenson$^{1}$
\and Lucas Kramer$^{4}$
\and Richard L. Kramer$^{1}$
\and Jephian Chin-Hung Lin$^{1}$
\and Ryan R. Martin$^{1,5}$
\and Derrick Stolee$^{1,6}$
\and Nathan Warnberg$^{7}$
 \and Michael Young$^{1,8,9}$}
\begin{document}
\maketitle

\footnotetext[1]{Department of Mathematics, Iowa State University, Ames, IA 50011, USA. \newline{\{butler, lhogben, kahogens, ricardo, chlin, rymartin, dstolee, myoung\}@iastate.edu}}
\footnotetext[2]{Department of Mathematics and Computer Science, Grand View University, Des Moines, IA 50316, USA. (cerickson@grandview.edu)}
\footnotetext[3]{American Institute of Mathematics, 600 E. Brokaw Road, San Jose, CA 95112, USA (hogben@aimath.org).}
\footnotetext[4]{Mathematics Department, Bethel College, North Newton, KS 67117, USA. (lkramer@bethelks.edu)}
\footnotetext[5]{Research supported in part by National Security Agency grant H98230-13-1-0226.}
\footnotetext[6]{Department of Computer Science, Iowa State University, Ames, IA 50011, USA.}
\footnotetext[7]{Department of Mathematics and Statistics, University of Wisconsin-La Crosse, La Crosse, WI 54601, USA. (nwarnberg@uwlax.edu}
\footnotetext[8]{Research supported in part by NSF DMS 0946431.}
\footnotetext[9]{Corresponding author.}

%%%%%%%%%%%%%%%%%%%%%%%%%%%%%%%%%%%%%%%%%%%%%%%%%%%%
%%%%%%%%%%%%%%%%%%%%%%%%%%%%%%%%%%%%%%%%%%%%%%%%%%%%
%\linenumbers Line numbers

\begin{abstract}  In this paper, we investigate the anti-Ramsey (more precisely, anti-van der Waerden) properties of arithmetic progressions. For positive integers $n$ and $k$, the expression $\aw([n],k)$ denotes the smallest number of colors with which the integers $\{1,\ldots,n\}$ can be colored and still guarantee  there is a rainbow arithmetic progression of length $k$. We establish that $\aw([n],3)=\Theta(\log n)$ and $\aw([n],k)=n^{1-o(1)}$ for $k\geq 4$.

For positive integers $n$ and $k$, the expression $\aw(\Z_n,k)$ denotes the smallest number of colors with which elements of the cyclic group of order $n$ can be colored and still guarantee  there is a rainbow arithmetic progression of length $k$. In this setting, arithmetic progressions can ``wrap around,'' and %, in fact,
$\aw(\Z_n,3)$ behaves quite differently from $\aw([n],3)$, depending on the divisibility of $n$. As shown in [Jungi\'c et al.,  \textit{Combin. Probab. Comput.}, 2003],  $\aw(\Z_{2^m},3) = 3$ for any positive integer $m$.  We establish that $\aw(\Z_n,3)$ can be computed from knowledge of $\aw(\Z_p,3)$ for all of the prime factors $p$ of $n$.  However, for $k\geq 4$, the behavior is similar to the previous case, that is, $\aw(\Z_n,k)=n^{1-o(1)}$.  
\end{abstract}\smallskip


%%%%%%%%%%%%%%%%%%%%%%%%%%%%%%%%%%%%%%%%%%%%%%%%%%%%

\section{Introduction}\label{sec:intro}

Let $G$ be an additive (abelian) group such as the integers or the integers modulo $n$, and let $S$ be a finite nonempty subset of $G$.
A {\em $k$-term arithmetic progression ($k$-AP)} in $S$ is a set of distinct elements of the form \[a,a+d,a+2d,\dots,a+(k-1)d\] where $d\geq 1$ and $k\geq 2$.    An {\em $r$-coloring} of  $S$ is a function $c:S\rightarrow [r]$, where $[r]:=\{1,\dots,r\}$. We say such a coloring is {\em exact} if $c$ is surjective.  Given an $r$-coloring $c$ of $S$, the {\em $i^{\it th}$  color class} is $C_i:=\{x\in S: c(x)=i\}$.
An arithmetic progression is called {\em rainbow} if the image of the progression under the $r$-coloring is injective. Formally, given $c:S\rightarrow [r]$ we say a $k$-term arithmetic progression is rainbow if $\{c(a+id) : i=0,1,\ldots,k-1\}$ has $k$ distinct values.

The {\em anti-van der Waerden number} $\aw(S,k)$ is the smallest $r$ such that every exact $r$-coloring of $S$ contains a rainbow $k$-term arithmetic progression.  Note that this tautologically defines $\aw(S,k)=|S|+1$ whenever $|S|<k$, and this definition retains the property that there is a coloring with  $\aw(S,k)-1$ colors that has no rainbow $k$-AP. \ Since $\aw(S,2)=2$ for all $S$, we assume henceforth that $k\ge 3$.  

Several important results on the existence of rainbow 3-APs implying information about $\aw([n],3)$ and $\aw(\Z_n,3)$ (in our notation) have been established by Jungi\'c, et al.~\cite{JLMNR03}.   A preliminary study of the anti-van der Waerden
number was done by Uherka  in \cite{U13}; it should be noted the notation there is slightly different, with $AW(k,n)$ used to denote our $\aw([n],k)$. Other results    on balanced colorings of the integers with no rainbow $3$-AP have been obtained  by Axenovich and Fon-Der-Flaass~\cite{AF04} and Axenovich and Martin~\cite{AM06}.

First, we consider the set $S=[n]$. The value of $\aw([n],3)$ is logarithmic in $n$:
\begin{thm}\label{RM1} For every integer $n\geq 9$,
	\[\left\lceil\log_3n\right\rceil+2\le \aw([n],3) \le \left\lceil\log_2 n \right\rceil +1 .\]
Moreover, $\aw([n],3)=\left\lceil\log_3n\right\rceil+2$ for $n\in\{3,4,5,6,7\}$ and $\aw([8],3)=5$. \end{thm}
Theorem \ref{RM1} is proven by Lemmas~\ref{awn3log3lowerbnd} and~\ref{LHprop220nolab}  (for $n\ge 9$), and
 Remark~\ref{RK:obs:values} gives exact values of $\aw([n],3)$ that justify the second statement.  We conjecture that the lower bound is, essentially, correct:
\begin{conj}\label{RMconj} There exists a constant $C$ such that $\aw([n],3)\leq\left\lceil\log_3n\right\rceil+C$ for all $n\geq 3$. \end{conj}

The behavior of $\aw([n],k)$ is, however, different for $k\geq 4$. Instead of logarithmic, it is almost linear:
\begin{thm}\label{RM2}  For $k\ge 4$,
\[
ne^{-O\left(\sqrt{\log n}\right)}<\aw([n],k)\leq ne^{-\log\log\log n-\omega(1)}.
\]
\end{thm}
Theorem~\ref{RM2} is established by Lemma~\ref{CEthm:Behrendbound} and Corollary~\ref{CEthm:upawboundone}.

Finally, we consider arithmetic progressions in the cyclic group $\Z_n$.
\begin{rem}\label{KH:Thm:ZnVSBracketn}{For positive integers $n$ and $k$, $\aw(\Z_n,k) \leq \aw([n],k)$, because every AP in $[n]$ corresponds to an AP in $\Z_n$.   }\end{rem}

However, because progressions in $\Z_n$ may ``wrap around,'' there are additional APs in $\Z_n$, some of which may be rainbow.  Thus it is possible that every coloring of $\Z_n$ with $\aw([n],k)-1$ colors guarantees a rainbow $k$-AP, so strict inequality is possible. As was shown in \cite[Theorem 3.5]{JLMNR03} (and follows from  Theorem~\ref{RM4} below), there are infinitely many values of $n$ for which $\aw(\Z_n,3)=3$, for example, when $n$  is a power of two.

\begin{defn}{Let $n\ge 3$ be an integer. % with prime  factorization  $n=2^{\ell_0}p_1^{\ell_1}\cdots p_t^{\ell_t}$.
Define $f_2(n)$ to be 0 if $n$ is odd and  1 if $n$ is even.  Define $f_3(n)$ to be the number of  odd prime factors $p$ of $n$  that have $\aw(\Z_{p},3)=3$ and  $f_4(n)$ to be the number of  odd prime factors $p$ of $n$  that have $\aw(\Z_{p},3)=4$, both counted according to  multiplicity. }
\end{defn}

\begin{thm}\label{RM4}  
For every prime number $p$, $3\le \aw(\Z_p,3)\le 4$.  For an  integer $n\ge 2$, the value of $\aw(\Z_{n},3)$  is determined by the values of $\aw(\Z_p,3)$ for the prime factors $p$:	%\beq\label{LH:exact3} 
\[\aw(\Z_n,3) = 2+f_2(n)+f_3(n)+2f_4(n).\]
For an  integer $n\ge 2$ having every prime factor  less than $100$,
 $f_4(n)$ is the number of odd prime factors of $n$ in the set $Q_4:=\{17, 31, 41, 43, 73, 89,  97\}$ and $f_3(n)$ is the number of odd prime factors of $n$ in $Q_3$, where $Q_3$ is the set of all odd primes less than $100$ and not in $Q_4$.   %      \een
\end{thm}
Theorem \ref{RM4} is established by %Lemma~\ref{CH:thm:awz2^n} and 
Proposition \ref{newsingleton}, 
Corollary~\ref{LH:Zn3primefactor3} and Proposition \ref{LH:Zn3primeval}.

For $k\geq 4$, the bounds we obtain for $\aw(\Z_n,k)$ are the same as those for $\aw([n],k)$:
\begin{thm}\label{RM5}  For $k\ge 4$,
\[
    ne^{-O\left(\sqrt{\log n}\right)}< \aw(\Z_n,k) \leq ne^{-\log\log\log n - \omega(1)}.
\]
\end{thm}
Theorem \ref{RM5} is established by %Theorem \ref{RM1}, 
Remark \ref{KH:Thm:ZnVSBracketn} and  Lemma~\ref{KH:Thm:BehrendtypelbZn}.

The structure of the paper is as follows: Section \ref{sec:awnk} presents results pertaining to $\aw([n],k)$, with Theorem \ref{RM1} proved in Section \ref{ssec:awn3} and Theorem \ref{RM2} proved in Section \ref{ssec:awnk}. Results pertaining to $\aw(\Z_n,k)$ appear in Section \ref{sec:awZnk}, with Theorem \ref{RM4} proved in Section \ref{ssec:awZn3} and Theorem \ref{RM5} proved in Section \ref{ssec:awZnk}.  Section \ref{sec:compute} describes the methods and algorithms used to compute values of $\aw([n],k)$ and $\aw(\Z_n,k)$, while Section \ref{sec:conclude} contains conjectures and open questions for future research.

In the remainder of this section we establish a basic but necessary observation that $\aw(S,\cdot)$ is monotone in $k$.

\begin{obs}\label{thm:awnmonok} %{\rm \cite{U13}}
Let $G$ be an additive (abelian) group such as the integers or the integers modulo $n$,  let $S$ be a finite nonempty subset of $G$, and let  $k\ge 3$ be an integer. Then $\aw(S,k)\leq \aw(S,k+1)$.
\end{obs}

Observation~\ref{thm:awnmonok} follows immediately from Proposition~\ref{thm:ezlem} below and was noted
noted by Uherka in {\cite{U13}} for the function $\aw([n],\cdot)$.

\begin{prop}\label{thm:ezlem}
	Let $G$ be an additive (abelian) group such as the integers or the integers modulo $n$,  let $S$ be a finite nonempty subset of $G$, and let  $k\ge 3$ be an integer.   If there is an exact $r$-coloring of $S$ that has no rainbow $k$-AP, then $\aw(S,k) \ge r+1$.
	%nonstandard use of $r$
\end{prop}

\bpf
Let $c$ be an exact $r$-coloring of $S$ with color set $\{1,\ldots,r\}$ that has no rainbow $k$-AP. We proceed by constructing an exact $(r-1)$-coloring of $S$ with no rainbow $k$-AP. For $x\in S$, define

\[\hat{c}(x)=\left\{
\begin{array}{lll}
c(x) & \mbox{if} & c(x)\in\{1,\ldots,r-2\},\\
r-1 & \mbox{if} & c(x)\in\{r-1,r\}.
\end{array} \right. \]
Note that $\hat{c}$ is an exact $(r-1)$-coloring of $S$. Let $K$ be a $k$-AP in $S$. Since there is no rainbow $k$-AP under $c$ there exists $j,\ell\in K$ such that $c(j)=c(\ell)$. It then follows that $\hat{c}(j)=\hat{c}(\ell)$. Hence $K$ is not rainbow under the coloring $\hat{c}$. By the generality of $K$, $\hat{c}$ is an exact $(r-1)$-coloring of $S$ that has no rainbow $k$-AP.  Repeating this construction we  obtain an exact $(r-i)$-coloring of $S$ with no rainbow $k$-AP for $i\in\{1,2,\ldots,r-1\}$. Therefore $\aw(S,k)\geq r+1$.
\epf

%%%%%%%%%%%%%%%%%%%%%%%%%%%%%%%%%%%%%%%%%%%%%%%%%%%%

\section{$\aw([n],k)$}\label{sec:awnk}%$\null$
%{\blue Ryan will edit}

 In this section we establish properties of $\aw([n],k)$.   Sections 2.1 and 2.2 establish our main results for $\aw([n],3)$ and $\aw([n],k), k\ge 4$, respectively.  Sections 2.3 and 2.4 contain additional results valid for all $k$ and specific to $k=3$, respectively.

In Table~\ref{tab:nallk} we give our calculated values of $\aw([n],k)$ for $k\ge 3$. We have a larger list of known values in the case of $k=3$ that is included in Remark~\ref{RK:obs:values} below; in Table~\ref{tab:nallk} we include  only the values $\aw([n],3)$ for which we have a value for $\aw([n],4)$ so that we may compare them.   We also restrict $n,k\geq 3$, and have stopped with $k= \left\lceil\frac{n}{2}\right\rceil+1$, because $\aw([n],k)=n$ if and only if $k\geq \left\lceil\frac{n}{2}\right\rceil+1$ (Proposition~\ref{LHawnkequalsnthm} below).

The growth rates when $k=3$ and when $k\geq 4$ appear to be different based on data given in Table~\ref{tab:nallk}. %With this in mind we separate our discussion of $\aw([n],k)$ such that $k=3$ and $k\geq 4$ into Sections \ref{ssec:awn3} and  \ref{ssec:awnk}.
The upper bound of $\left\lceil\log_2 n\right\rceil+1$ given in Proposition~\ref{LHprop220nolab} for $k=3$ and the lower bound of $n^{1-o(1)}$ in Lemma~\ref{CEthm:Behrendbound} for $k\geq 4$ confirm that the growth rates are indeed radically different. % that the cases of $k=3$ and $k\geq 4$ have different growth rates.

\begin{table}[htp]
\centering
{\small
\begin{tabular}[h]{c|ccccccccccccc}
$n\setminus k$ & 3 & 4 & 5 & 6 & 7 & 8 & 9 & 10 & 11 & 12 & 13 & 14 \\\hline
3 & 3\\
4 & 4 \\
5 & 4 & 5\\
6 & 4 & 6 \\
7 & 4 & 6 & 7 \\
8 & 5 & 6 & 8 \\
9 & 4 & 7 & 8 & 9\\
10 & 5 & 8 & 9  & 10 \\
11 & 5 & 8 & 9 & 10 & 11\\
12 & 5 & 8 & 10 & 11 & 12\\
13 & 5 & 8 & 11 &11 & 12 & 13\\
14 & 5 & 8 & 11 & 12 & 13 & 14\\
15 & 5 & 9 & 11 & 13 &14 & 14 & 15\\
16 & 5 & 9 & 12 & 13 & 15 & 15& 16\\
17 & 5 & 9 & 13 & 13& 15 & 16 & 16 & 17\\
18 & 5 & 10 & 14 & 14 & 16 & 17 & 17 & 18 \\
19 & 5 & 10 & 14 & 15 & 17 & 17 & 18 & 18  & 19 \\
20 & 5 & 10 & 14 & 16 & 17 & 18  & 19 & 19 & 20\\
21 & 5 & 11 & 14 & 16 & 17 & 19 & 20 & 20 & 20 & 21 \\
22 & 6 & 12 & 14 & 17 & 18 & 20 & 21 & 21 & 21 & 22\\
23 &  6 & 12 & 14 & 17 & 19 & 20 & 21 & 22 &22 & 22 & 23\\
24 & 6 & 12 & 15 & 18 & 20 & 20 & 22 & 23 & 23 & 23 &24 \\
25 & 6 & 12 & 15 & 19 & 21 & 21 & 23 & 23 & 24 & 24 & 24 & 25  \\
%26 & 6 & 12 &
\end{tabular}
}
\caption{\label{tab:nallk}Values of $\aw([n],k)$ for $3 \leq k \leq \frac{n+3}{2}$.}
\end{table}

%------------------------------------------------------------------------------------------
%\pagebreak
\subsection{Main results for $\aw([n],3)$}\label{ssec:awn3}

Before we address Theorem~\ref{RM1}, we show a summary of the computed data for this case in Remark~\ref{RK:obs:values} below.

\begin{rem}\label{RK:obs:values}{
The exact values of $\aw([n],3)$ are known from computer computations (described in Section~\ref{sec:compute})  for $n\leq58$, and are recorded here.
\ben
\item {$\aw([n],3)=2$ for $n\in\{1\}$.}
\item
$\aw([n],3)=3$ for $n\in\{2,3\}$.
\item
$\aw([n],3)=4$ for $n\in\{4,\ldots,7\}\cup\{9\}$.
\item
$\aw([n],3)=5$ for $n\in\{8\}\cup\{10,\ldots,21\}\cup\{27\}$.
\item
$\aw([n],3)=6$ for $n\in\{22,\ldots,26\}\cup\{28,\ldots,58\}$.
\een
}\end{rem}

Now we turn to the proof of Theorem~\ref{RM1}, beginning with the lower bound.

%\subsubsection{Theorem~\ref{RM1}: Proof of lower bound}

\begin{prop}
\label{RK:thm:lwrawboundoneforkequalsthree}
Let $n$ be a positive integer and let  $s\in\{-2,-1,0,1,2\}$. Then $\aw([3n-s],3)\geq\aw([n],3)+1$ provided $n\ge s$.
\end{prop}
\bpf
Let $r=\aw([n],3)$ and $s\in\{0,1,2\}$. We construct an exact $r$-coloring of $[3n-s]$ that does not contain a rainbow $3$-AP. By definition there exists an exact $(r-1)$-coloring, denoted $c$, of $[n]$ such that there is no rainbow $3$-AP in $[n]$.
Color $[3n-s]$ in the following manner: If $i+s$ is divisible by 3, then $\hat c(i)=c((i+s)/3)$, otherwise $\hat c(i)=r$.
Consider a $3$-AP, $K$, in $[3n-s]$. Then either the three terms in $K+s$ are all divisible by 3 or at least two of the terms in $K+s$ are not divisible by 3. If all terms in $K+s$ are divisible by 3, then $K$ is not rainbow under $\hat c$, since there is no rainbow 3-AP under $c$. If two terms of $K+s$ are not divisible by 3 then those two terms are both colored $r$ and $K$ is not rainbow.  Hence $\aw([3n-s],3)\geq r+1$ for $s\in\{0,1,2\}$.

For $s\in\{-2,-1\}$,  use the same coloring as for $s=0$.
\epf

Using Proposition~\ref{RK:thm:lwrawboundoneforkequalsthree}, we establish the lower bound in Theorem~\ref{RM1}.

\begin{lem}\label{awn3log3lowerbnd}
Let $n$ be a positive integer. Then $\aw([n],3)\geq\left\lceil\log_3n\right\rceil+2$.
\end{lem}

\bpf
The proof is by induction.  The cases $n=1,2,3$ are true by inspection. Suppose $n>3$ and that $\aw([m],3)\geq\left\lceil\log_3m\right\rceil+2$ for all $m$ satisfying $1\leq m<n$.
We  show that $\aw([n],3)\geq\left\lceil\log_3n\right\rceil+2$.
First, we write $n=3m-s$, where $s\in\{0,1,2\}$ and $2\leq m<n$.
Then by Proposition~\ref{RK:thm:lwrawboundoneforkequalsthree}, \[\aw([n],3)=\aw([3m-s],3)\geq\aw([m],3)+1\geq\left\lceil\log_3m\right\rceil+2+1=\left\lceil\log_3(3m)\right\rceil+2\geq\left\lceil\log_3n\right\rceil+2.\qedhere\]
\epf

%In the next example we exhibit colorings with  .

\begin{ex}\label{LHj1}{Induction and the proof of Proposition~\ref{RK:thm:lwrawboundoneforkequalsthree} produce the following exact $(m+1)$-coloring of $[3^m]$ that does not have a rainbow 3-AP: For $x\in[3^m]$ with the prime factorization $x=2^{e_2}3^{e_3}5^{e_5}\cdots p^{e_p}$, $c(x)=m+1-e_3$.   This attains the value in Lemma~\ref{awn3log3lowerbnd}.}
\end{ex}

To complete the proof of Theorem~\ref{RM1}, we establish the upper bound.

%We now return to establishing upper bounds.

\begin{lem}\label{RK:lem:powersoftwo}
Let $c$ be an exact $r$-coloring of $[n]$  that does not have a rainbow $3$-AP. For $i\in[r]$, define $b_i\in[n]$ to be the least $x$ such that the induced coloring on $[x]$ has exactly $i$ colors.  Then for all $i\in[r-1]$,  $b_{i+1}\geq 2b_i$. Furthermore, for any $1\leq i\leq j\leq r$, we have
$b_{j}\geq 2^{j-i}b_i$.
\end{lem}

\bpf  Observe that $b_1=1$.
Suppose that $b_{i+1}< 2b_i$ for some $i\in[r-1]$. Then $\{2b_i-b_{i+1},b_i,b_{i+1}\}$ is a rainbow $3$-AP. The last statement follows by induction, since $b_i\leq2^{-1}b_{i+1}\leq2^{-2}b_{i+2}\leq\dots\leq
2^{-(j-i)}b_{j}$.
\epf

%Using Lemma~\ref{RK:lem:powersoftwo}, we establish the upper bound in Theorem~\ref{RM1}.

\begin{lem}\label{LHprop220nolab}
	For $n \ge 9$, $\aw([n],3) \le \left\lceil\log_2 n \right\rceil +1$.
\end{lem}

\bpf  Suppose $r=\aw([n],3)-1$, so there is an $r$-coloring with no rainbow 3-AP.  Lemma \ref{RK:lem:powersoftwo} implies that $n\ge b_r\ge 2^{r-1}$.  Thus $\aw([n],3) \le \left\lfloor\log_2 n \right\rfloor +2$, which establishes the result for $n$ not a power of 2.  The case $\aw([2^m],3)\le m+1$ follows similarly by using the fact that $\aw([2^m],3) = m +1$ for $m=4$ and $m=5$ (see Remark \ref{RK:obs:values}); $\aw([16],3) =5$ implies $b_5>16=2^4$ for any rainbow-free coloring with $r\ge 5$.  Then for $m>5$, Lemma \ref{RK:lem:powersoftwo} implies an $r$-coloring of $2^m$   has $2^m\ge b_r\ge 2^{r-5}b_5> 2^{r-1}$, so $m\ge r$ and $m+1\ge \aw(2^m,3)$. 
\epf

This completes the proof of Theorem \ref{RM1}.

%------------------------------------------------------------------------------------------
\subsection{Main results for $\aw([n],k), k\ge 4$}\label{ssec:awnk}

In this section we specialize to the case $k\ge 4$, focusing on   lower and upper  bounds that give $\aw([n],k)=n^{1-o(1)}$.  Lemma~\ref{CEthm:Behrendbound} gives the lower bound and Corollary~\ref{CEthm:upawboundone} gives the upper bound.

Let $\sz(n,k)$ denote the largest size of a set $S \subseteq [n]$ such that $S$ contains no $k$-AP (similar notation was introduced in \cite{gasarch2008finding} in honor of Szemer\'edi~\cite{Sz}).
Determining bounds on $\sz(n,k)$ is a fundamental problem in the study of arithmetic progressions. Behrend \cite{B46}, Gowers  \cite{Gow01}, and others \cite{LL01, R61} have established various bounds on $\sz(n,k)$.
Proposition~\ref{prop:snk} provides a strong link between $\sz(n,k)$ and our anti-van der Waerden numbers, allowing us to use known results on $\sz(n,k)$ to bound $\aw(n,k)$.

\begin{prop}\label{prop:snk}
For all $n > k \geq 3$,
\[
	\sz(n, \lfloor k/2\rfloor) + 1 \leq \aw([n], k) - 1 \leq \sz(n,k).
\]
\end{prop}

\begin{proof}
If $c$ is an exact $r$-coloring of $[n]$ that contains no rainbow $k$-AP, then selecting one element of each color class creates a set $S$ that contains no $k$-AP; therefore $\aw([n],k) - 1 \leq \sz(n,k)$.
If $S$ is a set in $[n]$ that contains no $\lfloor k/2\rfloor$-AP, then color $[n]$ by giving each element in $S$ a distinct color and the elements of $[n]\setminus S$ a new color.  %, called \emph{zero}.
If a $k$-AP $\{a_1, a_2,\dots, a_k\}$ is rainbow in this coloring, then exactly one such $a_i$ is in $[n] \setminus S$.
But this implies that the entries $a_j$ where $j \not\equiv i \pmod 2$ form an AP in $S$ with at least $\lfloor k/2\rfloor$ terms, a contradiction.
\end{proof}

\def\vx{\mathbf{x}}
\def\vy{\mathbf{y}}

\subsubsection{Theorem~\ref{RM2}: Proof of lower bound}

Lemma~\ref{LHprop220nolab} and Behrend's results (stated in Theorem~\ref{lma:behrendsize} and Proposition~\ref{lma:behrend3ap} below) show that the upper bound in Proposition \ref{prop:snk} is not useful for $k=3$.  Observe that when $k\in \{4,5\}$, the lower bound in Proposition \ref{prop:snk} is trivial but is in fact useful in the case of $k\geq 6$.
We provide a similar lower bound for $k \in \{4,5\}$ in Lemma~\ref{CEthm:Behrendbound} by carefully studying Behrend's original construction~\cite{B46} of a relatively large set $S \subset [n]$ that contains no 3-AP, thus giving a lower bound on $\sz(n,3)$.

Let $\{a_1, a_2, a_3, a_4\}$ be a 4-AP with $a_1 \le a_2 \le a_3 \le a_4$.
A set $A \subset \{a_1, a_2, a_3, a_4\}$ of size $|A| = 3$ is called a \emph{punctured 4-AP}.
If such a punctured 4-AP $A$ is not a 3-AP, then it is of the form $A = \{a_1, a_2, a_4\}$ or $A = \{a_1, a_3, a_4\}$. We prove that Behrend's construction in fact contains no punctured 4-AP (Proposition~\ref{lma:punctured4ap} below).
This leads to Lemma~\ref{CEthm:Behrendbound} below.

\begin{lem}\label{CEthm:Behrendbound}
There exists an absolute constant $b > 0$ such that for all $n, k \geq 4$,
\[\aw([n],k) >  n e^{-b\sqrt{\log n}}= n^{1-o(1)}.\]
\end{lem}

The proof of Lemma~\ref{CEthm:Behrendbound} follows from Proposition~\ref{lma:punctured4ap}, Theorem~\ref{lma:behrendsize}, Proposition~\ref{lma:behrend3ap} and Proposition~\ref{lma:nopunctured4ap}, which follow.

\begin{prop}\label{lma:punctured4ap}
Suppose $S \subseteq [n]$ does not contain any punctured 4-APs.
Then $\aw([n], k) > |S|+1$ for all $n\ge k \geq 4$.
\end{prop}

\begin{proof}
Color each member of $S$ a distinct color, and color each integer in $[n] \setminus S$ with a new color called {\em zero}.
If there is a rainbow 4-AP in this coloring, then at most one of the elements in this 4-AP is colored zero.
Thus there must be a punctured 4-AP in the other colors, but $S$ contains no punctured 4-AP.
\end{proof}

There is a bijection between vectors $\vx = (x_1,\dots,x_m)^\top \in \Z^m$ where $x_i \in \{0,1,\dots, 2d-2\}$ for all $i \in [m]$ and elements of $\{0,1,\dots,(2d-1)^m-1\}$, by viewing $\vx$ as a $(2d-1)$-ary representation of an integer:
\[
	\vx = (x_1,\dots,x_m)^\top \longleftrightarrow a_{\vx} = \sum_{i=1}^m x_i (2d-1)^{i-1}.
\]
Moreover, observe that if $\vx, \vy \in  \Z^m$ with $x_i, y_i \in \{0,\dots, d-1\}, i=1,\dots,m,$ are associated with $a_\vx, a_\vy \in \{0,1,\dots,(2d-1)^m-1\}$ by this bijection, then $\vx+\vy$ has $x_i+y_i\in\{0,\dots, 2d-1\}$ and $\vx+\vy$ is associated with $a_{\vx+\vy}=a_{\vx}+a_{\vy}\in \{0,1,\dots,(2d-1)^m-1\}$.

Recall that for a vector $\vx\in\R^m$, $||\vx||^2 = \sum_{i=1}^m x_i^2$.  Let $m, \ell, d$ be positive integers and define $X_\ell(m,d)$ to be the set of vectors $\vx = (x_1,\dots,x_m)^\top$ such that
\begin{enumerate}
\item $x_i \in \{0,\dots, d-1\}$ for all $i \in \{1,\dots,m\}$, and
\item $||\vx||^2 =  \ell$.
\end{enumerate}
The set $S_\ell(m,d)$ of integers associated with  the vectors in $X_\ell(m,d)$ via the map $\vx\to a_{\vx}$ forms a subset of integers in $\{0,1,\dots, (2d-1)^m-1\}$.
Behrend~\cite{B46} used the pigeonhole principle to prove the following lemma; here we state the version from \cite{E11}.

\begin{thm}\label{lma:behrendsize}{\rm \cite{B46, E11}}
There exist absolute constants $b, b' > 0$ such that for all $n$ and positive integers $m = m(n)$, $\ell = \ell(n)$, and $d = d(n)$ such that $S_\ell(m,d) \subseteq [n]$ the following inequality holds:
\[|S_\ell(m,d)| \geq \frac{b'n}{2^{ \sqrt{8\log_2 n}}(\log n)^{1/4}} \geq n e^{-b\sqrt{\log n}}.\]
\end{thm}

The important property of $S_\ell(m,d)$ is that it avoids non-trivial arithmetic progressions.
We include Behrend's simple proof of this fact for completeness.

\begin{prop} {\rm \cite{B46}}\label{lma:behrend3ap}
The set $S_\ell(m,d)$ contains no 3-AP.
\end{prop}

\begin{proof}
Suppose $\{a_{\vx_1}, a_{\vx_2}, a_{\vx_3}\}$ is a 3-AP in $S_\ell(m,d)$.
Let $\vx_1, \vx_2, \vx_3$ be the associated vectors in $X_\ell(m,d)$.
Since $a_{\vx_1} + a_{\vx_3} = 2a_{\vx_2}$, we also have that $\vx_1 + \vx_3 = 2\vx_2$. See Figure~\ref{fig1}.
However, by the triangle inequality, we have that
\[
	2 \sqrt{\ell} = 2||\vx_2|| = || \vx_1 + \vx_3 || \leq || \vx_1|| + ||\vx_3|| = 2\sqrt{\ell},
\]
and equality can only hold if $\mathbf 0$, $\vx_1$, $\vx_3$ and $2\vx_2$ are collinear.
However, since $||\vx_1|| = ||\vx_3||$, this would imply $\vx_1 = \vx_3$ and thus $a_{\vx_1} = a_{\vx_3}$, a contradiction.
\end{proof}

\begin{figure}[htp]
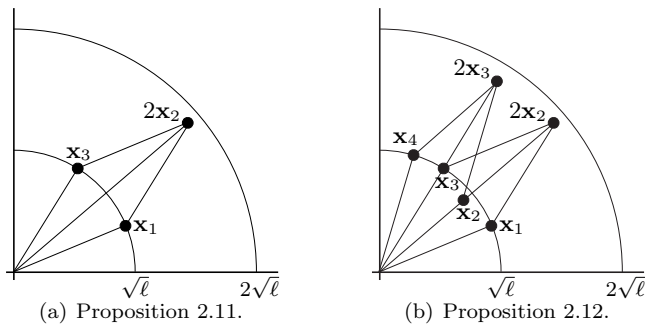

\centering
\mbox{
\subfigure[Proposition~\ref{lma:behrend3ap}.]{
\begin{lpic}[]{"behrend1"(,1.5in)}
\lbl[b]{27,42;\small $\vx_3$}
\lbl[br]{61,55;\small $2\vx_2$}
\lbl[tl]{45,22;\small $\vx_1$}
\lbl[t]{45.6,3.5;\footnotesize $\sqrt{\ell}$}
\lbl[t]{87,3.5;\footnotesize $2\sqrt{\ell}$}
\end{lpic}
}
\qquad
\subfigure[Proposition~\ref{lma:nopunctured4ap}.]{
\begin{lpic}[]{"behrend3"(,1.5in)}
\lbl[b]{14,47;\small $\vx_4$}
\lbl[b]{28,32;\small $\vx_3$}
\lbl[br]{42.5,69.5;\small $2\vx_3$}
\lbl[t]{35,26;\small $\vx_2$}
\lbl[br]{61,55;\small $2\vx_2$}
\lbl[tl]{45,22;\small $\vx_1$}
\lbl[t]{45.6,3.5;\footnotesize $\sqrt{\ell}$}
\lbl[t]{87,3.5;\footnotesize $2\sqrt{\ell}$}
\end{lpic}
}
}
\caption{Proofs of Propositions~\ref{lma:behrend3ap} and~\ref{lma:nopunctured4ap}.}\label{fig1}
\end{figure}

\begin{prop}\label{lma:nopunctured4ap}
The set $S_\ell(m,d)$ contains no punctured 4-AP.
\end{prop}

\begin{proof}
%Suppose $\{a_{\vx_1}, a_{\vx_2}, a_{\vx_3}\}$ is a 3-AP in $S_\ell(m,d)$.
Let $\{a_{\vx_1}, a_{\vx_2}, a_{\vx_3}, a_{\vx_4}\}$ be a 4-AP.
Since $S_\ell(m,d)$ contains no 3-AP, it must be that one of $a_{\vx_2}$ or $a_{\vx_3}$ is not in $S_\ell(m,d)$.
Assume by symmetry that $a_{\vx_2} \notin S_\ell(m,d)$ and $a_{\vx_1}, a_{\vx_3}, a_{\vx_4} \in S_\ell(m,d)$.
Let $\vx_1, \vx_2, \vx_3, \vx_4$ be the associated vectors where $\vx_1, \vx_3, \vx_4 \in X_\ell(m,d)$.

Since $a_{\vx_1} + a_{\vx_3} = 2a_{\vx_2}$, we have $\vx_1 + \vx_3 = 2\vx_2$. See Figure~\ref{fig1}.
However, as in the proof of Proposition~\ref{lma:behrend3ap}, this implies that $|| \vx_2 || < \sqrt{\ell}$.
Since $a_{\vx_2} + a_{\vx_4} = 2a_{\vx_3}$, we have $\vx_2 + \vx_4 = 2\vx_3$.
However, this implies that
\[
	2\sqrt{\ell} = 2||\vx_3|| = ||\vx_2 + \vx_4|| \leq ||\vx_2|| + ||\vx_4|| < 2\sqrt{\ell},
\]
a contradiction.
\end{proof}

Lemma~\ref{CEthm:Behrendbound} now follows by combining Propositions~\ref{lma:punctured4ap} and~\ref{lma:nopunctured4ap}.
It may be possible that the bound in Lemma~\ref{CEthm:Behrendbound} could be improved by using the construction of Elkin~\cite{E11, GW10} that avoids 3-APs using $\frac{b n(\log n)^{1/4}}{2^{\sqrt{8\log_2 n}}}$ elements for some constant $b > 0$. Since this construction avoids a $3$-AP, we can use Proposition~\ref{prop:snk} directly in order to obtain a coloring with no $6$-AP, giving $\aw([n],k) >  \frac{b n(\log n)^{1/4}}{2^{\sqrt{8\log_2 n}}}$ for all $k \geq 6$.
Further use of constructions of Rankin~\cite{R61} or Laba and Lacey~\cite{LL01} of large sets that avoid $k$-APs could slightly improve the asymptotics of $\aw([n],k)$, but these bounds are all of the form $n^{1-o(1)}$.

\subsubsection{Theorem~\ref{RM2}: Proof of upper bound}

A theorem of Gowers, stated here as Theorem \ref{CE:Gowersbound}, provides an upper bound for $\aw([n],k)$. However, $n$ must be very large compared to $k$ for this upper bound to be significantly different than the na\"{\i}ve upper bound of $n$ itself.

\begin{thm}\label{CE:Gowersbound}{\rm \cite[Theorem 1.3]{Gow01}}
For every positive integer $k$ there is a constant $b=b(k)>0$ such that every subset of $[n]$ of size at least $n(\log_2\log_2 n)^{-b}$ contains a $k$-AP. Moreover, $b$ can be taken to be $2^{-2^{k+9}}$\!.	
\end{thm}

\begin{cor}\label{CEthm:upawboundone}
Let  $n$ and $k$ be positive integers. Then there exists a constant $b$ such that $\aw([n],k)\leq\left\lceil n(\log_2\log_2 n)^{-b}\right\rceil$.
That is, for a fixed positive integer $k$, the function $\aw([n],k)$ of $n$ is $ o(\frac{n}{\log\log n})$.
\end{cor}

\bpf
Consider an exact $t$-coloring of $[n]$, where $t:=\left\lceil n(\log_2\log_2 n)^{-b}\right\rceil$ and $b=2^{-2^{k+9}}$\!. Since the coloring is exact, there exists a set $A\subseteq[n]$ of $t$ differently colored integers. By Theorem \ref{CE:Gowersbound}, $A$ contains a $k$-AP. Therefore $\aw([n],k)\leq t$.
\epf

Note that the upper bound in Corollary~\ref{CEthm:upawboundone} can be expressed as $ne^{-\log\log\log n-\omega(1)}$. Then combining this upper bound on $\aw([n],k)$ and the lower bound from Lemma~\ref{CEthm:Behrendbound}, we have that for $k\geq4$
\[
ne^{-b\sqrt{\log n}}<\aw([n],k)\leq ne^{-\log\log\log n-\omega(1)}.
\]
 This completes the proof of Theorem \ref{RM2}.

%------------------------------------------------------------------------------------------
%\pagebreak

\subsection{Additional results for $\aw([n],k)$ valid for all $k$}\label{ssec:awnall}

In this section we present some additional elementary results for $\aw([n],k)$.
The next proposition %Proposition~\ref{thm:awmonepone}
describes a relationship between $\aw([n],k)$ and $\aw([n-1],k)$.

\begin{prop}
	 Let $n$ and $k$ be positive integers.  Then $\aw([n],k) \le  \aw([n-1],k) +1$. \label{thm:awmonepone}
\end{prop}

\bpf
	Let $r = \aw([n],k)$. Note that  if $n< k$ our result follows from the definition. Suppose $n\ge k$. Then there is some exact $(r-1)$-coloring of $[n]$ that has no rainbow $k$-AP, and without loss of generality $n$ is colored $r-1$.  Consider this  coloring restricted to $[n-1]$.  Then we have two cases:
	\begin{enumerate}
		\item This is an exact $(r-1)$-coloring of $[n-1]$.
		\item The only integer in $[n]$ with the color $r-1$ is $n$, so this is an exact $(r-2)$-coloring of $[n-1]$.
	\end{enumerate}
	
	Note that since $[n]$ had no rainbow $k$-AP in both of our cases we still do not have a rainbow $k$-AP.  So by Proposition \ref{thm:ezlem} we have $\aw([n-1],k) \ge r-1 = \aw([n],k) -1$ and the result follows.
\epf

%As observed by Uherka \cite{U13} and illustrated in our computed values  in Table \ref{tab:nallk}, it is possible to have $\aw([n],k)< \aw([n-1],k)$, e.g., $\aw([9],3)< \aw([8],3)$.
 In the next proposition %Theorem~\ref{LHawnkequalsnthm}
 we characterize the values of $k$ for which $\aw([n],k)=n$.

\begin{prop}\label{LHawnkequalsnthm}
Let $n$ and $k$ be positive integers with $k\le n$. Then $\aw([n],k)=n$ if and only if $k\geq \left\lceil\frac{n}{2}\right\rceil+1$.
\end{prop}

\bpf
%It suffices to prove the statement for $k\geq \left\lceil\frac{n}{2}\right\rceil+1$ since $k$ must be an integer.

Suppose $k\geq \left\lceil\frac{n}{2}\right\rceil+1$. %Observe that any arithmetic progression in $[n]$ must be made up of consecutive integers.
We show that $\aw([n],k)>n-1$. %Color $[n]$ in the following fashion.
Color $\left\lceil \frac{n}{2} \right\rceil$ and $\left\lceil \frac{n}{2} \right\rceil+1$ with the same color and all the remaining integers with unique colors. This is an exact $(n-1)$-coloring. Since $k\geq \left\lceil \frac{n}{2} \right\rceil+1$, the integers in any $k$-AP must be consecutive integers, and the values $\left\lceil \frac{n}{2} \right\rceil$ and $\left\lceil \frac{n}{2} \right\rceil+1$ must be contained in any $k$-AP. Hence no $k$-AP is rainbow.

For the converse, suppose $\aw([n],k)=n$. Color $[n]$ with $n-1$ colors such that there is no rainbow $k$-AP.
Therefore exactly one color class has size two and the rest have size one.
%Hence all color classes are of size one except there is one class of size 2.
Denote the color class of size two as $C=\{n_1,n_2\}$, $n_1<n_2$. Then every $k$-AP contains both $n_1$ and $n_2$, or else we would have a rainbow $k$-AP. Suppose that $k\leq \left\lceil\frac{n}{2}\right\rceil$. Then $\{1,2,\ldots,k\}$ and $\{n-k+1, n-k+2,\ldots,n\}$ are $k$-APs. Note that $\{1,2,\ldots,k\}\subseteq\{1,2,\ldots,\left\lceil\frac{n}{2}\right\rceil\}$ and $\{n-k+1, n-k+2,\ldots,n\}\subseteq \{\left\lfloor\frac{n}{2}\right\rfloor+1, \ldots, n\}$. Then $n_1,n_2\in \{1,2,\ldots,\left\lceil\frac{n}{2}\right\rceil\}\cap \{\left\lfloor\frac{n}{2}\right\rfloor+1, \ldots, n\}$. This intersection is empty or contains one element depending on whether $n$ is even or odd. In both cases, this contradicts the fact that  $n_1\neq n_2$ and $n_1,n_2\in \{1,2,\ldots,\left\lceil\frac{n}{2}\right\rceil\}\cap \{\left\lfloor\frac{n}{2}\right\rfloor+1, \ldots, n\}$. Therefore $k\geq \left\lceil \frac{n}{2}\right\rceil +1$.
\epf

The following upper bound was proved by Uherka  {\cite{U13}}; we include the brief proof for completeness.

\begin{prop}\label{thm:awnsumoverns}{\rm \cite{U13}}
Let $n$, $k$, $n_1$, and $n_2$ be positive integers such that $k\leq n_1\leq n_2\leq n$ and $n_1+n_2=n$. Then $\aw([n],k)\leq \aw([n_1],k)+\aw([n_2],k)-1$.
\end{prop}
\bpf
Let $r =  \aw([n_1],k)+\aw([n_2],k)-1$, and consider an arbitrary exact $r$-coloring $c$  of
$[n]$. Let $r_1 = |c([n_1])|$ and $r_2 = |c(\{n_1 + 1,\dots, n_1 + n_2\})|$.  Since $n_1 +n_2 = n$,  $r\le r_1 +r_2$. This implies  that $r_1\ge\aw([n_1],k)$ or $r_2\ge\aw([n_2],k)$. Clearly $r_1\ge\aw([n_1],k)$ implies $c$ has a rainbow $k$-AP.   By translating $c(\{n_1 + 1,\dots, n_1 + n_2\})$ to a coloring on $[n_2]$, we also see that $c$ has a rainbow $k$-AP if $r_2\ge\aw([n_2],k)$.  Thus $\aw([n],k)\leq r= \aw([n_1],k)+\aw([n_2],k)-1$.  
\epf

%------------------------------------------------------------------------------------------
%\pagebreak
\subsection{Additional results for $\aw([n],3)$}\label{ssec:awn3more}

In this section we establish  additional  bounds on $\aw([n],3)$ in Propositions \ref{LHj2} and \ref{RK:prop:upper}, and use Proposition \ref{RK:prop:upper} together with Remark~\ref{RK:obs:values}, Proposition~\ref{RK:thm:lwrawboundoneforkequalsthree}, and Lemma~\ref{awn3log3lowerbnd}  to compute (at least) $93$ additional exact values for
$\aw([n],3)$.

\begin{prop}\label{LHj2}
For $n\geq 2$, there exists $m\le \lfloor\frac{n}{2}\rfloor$ such that $\aw([n],3)\leq \aw\left([m],3\right)+1$.
\end{prop}
\bpf%{RK:prop:upper}
We may assume that $n\geq 3$, since the case $n=2$ follows by inspection.
	Let $r = \aw([n],3)$.  Then there exists an $(r-1)$-coloring, namely $c$, of $[n]$ that has no rainbow 3-AP.   Let $t$ be the length of a shortest consecutive integer sequence in $[n]$ that contains all $r-1$ colors, say the interval is $\{s+1, s+2,\dots, s+ t\}$ for some $s$.  Define $\hat{c}$ to be an $(r-1)$-coloring of $[t] = \{1,2,\dots,t\}$ so that $ \hat{c}(j):=c(s+j) $ for $1 \le j \le t$.
Notice that $\hat{c}(1)$ and $\hat{c}(t)$ cannot be the same color and each must be the only element of its color class, or else we could find a smaller $t$.  Let $\hat{c}(1) = a$ and define $b_i$ to be the smallest element of $[t]$ such that $[b_i]$ has $i+1$ colors for $ 1\le i \le r-2$.  Note that if $b_i$ is odd, i.e., $b_i = 2x+1$, then $\{1,x+1,2x+1\}$ is a rainbow 3-AP.  So the set of even numbers of $[t]$ are colored with exactly $r-2$ colors with no rainbow 3-AP.
Define $m=\lfloor\frac{t}{2}\rfloor\le \lfloor\frac{n}{2}\rfloor $ and consider the coloring $\tilde c$ of $[m]$ induced by the coloring $\hat c$ of the even integers in $[t]$.  The coloring $\tilde c$ uses at least $r-2$ colors and has with no rainbow $3$-AP, so  $\aw([n],3) -1=(r-2)+1 \le \aw\left([m],3\right)$.
\epf

\begin{prop}\label{RK:prop:upper}
Let $m,$ $ n,$ and $\ell$ be positive integers. If $m< n<2^{\ell}(m+1)$, then $\aw([n],3)\leq\aw([m],3)+\ell$.
\end{prop}

\bpf
Suppose not. Then there exists $m,\ell\geq1$ and $n$ with $m< n<2^{\ell}(m+1)$ such that there is a   coloring $c$ on  $[n]$ using exactly $r=\aw([m],3)+\ell$ colors that does not have a rainbow $3$-AP.  For $i\in[r]$, let $b_i\in[n]$ be the least $x$ such that the induced coloring on $[x]$ has exactly $i$ colors. Since $r-\ell=\aw([m],3)$, we must have $b_{r-\ell}\geq m+1$, since otherwise the induced coloring on $[m]$ contains at least $\aw([m],3)$ colors, which is impossible. Thus by Lemma~\ref{RK:lem:powersoftwo}, $n\geq b_{r}\geq 2^{\ell}b_{r-\ell}
\geq 2^{\ell}(m+1)$,  which contradicts our assumption on $n$.
\epf

%We use Remark~\ref{RK:obs:values}, Lemma~\ref{awn3log3lowerbnd}, and Corollary \ref{RK:prop:upper}
\begin{cor}
$\aw([n],3)=7$ for $64\leq n\leq 80$.
\end{cor}

\bpf
Since $\aw([m],3)=6$ for $22\leq m\leq 26$, and $3\cdot22-2=64$ and
$3\cdot26+2=80$, we see that $\aw([n],3)\ge 7$, by Proposition~\ref{RK:thm:lwrawboundoneforkequalsthree}.
Since $27<64\leq n\leq 80<112=4\cdot28$, we have $\aw([n],3)\leq
\aw([27],3)+2=5+2=7$ by Proposition~\ref{RK:prop:upper} and Remark~\ref{RK:obs:values}.
\epf

\begin{cor}
$\aw([n],3)=7$ for $82\leq n \leq 111$.
\end{cor}

\bpf
Since $27<82\leq n\leq 111<112=4\cdot28$, we have $\aw([n],3)\leq
\aw([27],3)+2=5+2=7$ by Proposition~\ref{RK:prop:upper} and Remark~\ref{RK:obs:values}. Also, since $3^4=81<n\leq111 < 243=3^5$, we have
$4<\log_3 n\leq 5$, so that $\aw([n],3)\geq \left\lceil\log_3n\right\rceil+2=5+2=7$ by Lemma~\ref{awn3log3lowerbnd}.
\epf

\begin{cor}
$\aw([n],3)=8$ for $190\leq n\leq235$.
\end{cor}

\bpf
Since $\aw([m],3)=7$ for $64\leq m\leq 80$, and  $3\cdot64-2=190$ and
$3\cdot80+2=242$, we see that  $\aw([n],3)\ge 8$ for $190\leq n\leq242$, by Proposition~\ref{RK:thm:lwrawboundoneforkequalsthree}. % and Observation~\ref{RK:obs:lwrawboundoneforkequalsthree}.
For $58< n<236=2^2\cdot(58+1)$, we see that $\aw([n],3)\leq \aw([58],3)+2=6+2=8$, by Proposition~\ref{RK:prop:upper} and Remark \ref{RK:obs:values}.
\epf

Finally we combine the upper and lower bounds.

\begin{prop}\label{RK8thm}
If\, $3^u<n< 2\cdot 3^{u}+2$, then $u+3\le \aw([n],3)\le \aw([3^u],3)+1$.  If\, $2\cdot 3^{u}+1<n<  4\cdot 3^u+ 4$, %@@@@@@@@
% Referee change 9 -- is this okay? was $=3^{u+1}+3^u+4$,
%@@@@@@@@
then $u+3\le \aw([n],3)\le\aw([3^u],3)+2$.
\end{prop}

\bpf The lower bounds follow immediately from Lemma~\ref{awn3log3lowerbnd}, and the first upper bound follows immediately from Proposition~\ref{RK:prop:upper}.  For the second, apply Proposition~\ref{RK:prop:upper} with $m=2\cdot3^u+1$ and $\ell=1$ to obtain  $\aw([n],3)\le\aw([m],3)+1$ and since $3^u<m< 2\cdot 3^{u}+2$, $\aw([m],3)\leq\aw([3^u],3)+1$.
\epf

%%%%%%%%%%%%%%%%%%%%%%%%%%%%%%%%%%%%%%%%%%%%%%%%%%%%

\section{$\aw(\Z_n,k)$}\label{sec:awZnk}%

 In this section we establish properties of $\aw(\Z_n,k)$.   Sections \ref{ssec:awZn3} and \ref{ssec:awZnk} establish our main results for $\aw(\Z_n,3)$ and $\aw(\Z_n,k), k\ge 4$, respectively.  Section \ref{ssec:awZnkmore} contains additional results.  % and specific to $k=3$, respectively.

  Please note that for $x\in\Z$, we will also use $x$ to denote the equivalence class $\{x+in : i\in\Z\}$ in $\Z_n$.  Because arithmetic progressions may ``wrap around'' in the group $\Z_n$, we call attention to the fact  that we  consider only $k$-APs that include $k$ distinct members of $\Z_n$.
 Naturally, one of our first questions about $\aw(\Z_n,k)$ concerns its relationship with $\aw([n],k)$.  Theorem \ref{JLMNR3.5}(a) below and Lemma~\ref{awn3log3lowerbnd}  show that $\aw(\Z_n,k)$ need not be asymptotic to $\aw([n],k)$ for $k=3$ and $n=2^m$.  However, we do have the simple bound $\aw(\Z_n,k)\le \aw([n],k)$ (already stated in Remark \ref{KH:Thm:ZnVSBracketn}).

%------------------------------------------------------------------------------------------
%\pagebreak
\subsection{Main results for $\aw(\Z_n,3)$}\label{ssec:awZn3}

When we turn to the special case $k=3$, many values of $\aw(\Z_n,3)$ can be computed, and new phenomena appear. Our main results in this case are described by Theorem \ref{RM4}, which we establish in this section.

Currently available computational data is given in Table  \ref{tab:znk3}; the row label displays the range of $n$ for which the values of $\aw(\Z_n,3)$ are reported in that row, and the
column heading is the ones digit within this range.
This data led to the discovery of several results established in this section and is used to establish the second statement in Theorem \ref{RM4} that concerns integers  having all prime factors less than one hundred.
%Lemma~\ref{CH:thm:awz2^n} below shows that $\aw(\Z_{2^m},3) = 3$, illustrating that the value of $\aw(\Z_n,3)$ is still as low as possible for some arbitrarily large $n$. Corollary \ref{LH:Zn3primefactor3} below allows the computation of $\aw(\Z_n,3)$ from the prime factorization of $n$ for many $n$, including all $n$ for which every prime factor is less than 100. Finally, several interesting open questions remain.

\begin{table}[htp]
\centering
{\small
\begin{tabular}[h]{c|cccccccccc}
 & 0 & 1 & 2 & 3 & 4 & 5 & 6 & 7 & 8 & 9 \\
 \hline
 0--9 &  & & & 3 & 3 & 3 & 4 & 3 & 3 & 4\\
10--19 & 4 & 3 & 4 & 3 & 4 & 4 & 3 & 4 & 5 & 3\\
20--29 & 4 & 4 & 4 & 3 & 4 & 4 & 4 & 5 & 4 & 3\\
30--39 & 5 & 4 & 3 & 4 & 5 & 4 & 5& 3 & 4 & 4 \\
40--49 & 4 & 4 & 5 & 4 & 4& 5& 4 &3 & 4& 4\\
50--59 & 5 & 5& 4& 3& 6 & 4& 4& 4& 4& 3\\
60--69 & 5 & 3 & 5 & 5& 3 &4 &5 &3 &5 &4\\
70--79 & 5 & 3& 5& 4& 4& 5& 4& 4& 5& 3\\
80--89 & 4 & 6 & 5 & 3& 5& 5& 5& 4& 4& 4\\
90--99 & 6 & 4& 4& 5& 4& 4& 4 & 4& 5 & 5
\end{tabular}
}
\caption{\label{tab:znk3}Computed values of $\aw(\Z_n,3)$ for $n =3,\dots,99$ (the row label gives the range of $n$ and the
column heading is the ones digit within this range).}
\end{table}

Many odd primes  $p$ have $\aw(\Z_p,3)=3$ (see Table \ref{tab:znk3} above). However, there are several examples of odd primes $p$ for which $\aw(\Z_p,3)=4$. In Example \ref{CH:awZ17} below we exhibit an explicit exact coloring that establishes $\aw(\Z_{17},3)\geq 4$.

\begin{ex}\label{CH:awZ17}{Coloring the elements of $\Z_{17}$ in order as
    \[
         3 \ 1 \ 1 \ 2  \ 1\  2\  2 \ 2 \ 1\  1\  2 \  2  \  2  \ 1 \  2\  1\   1
    \]
is an exact $3$-coloring that does not contain a rainbow $3$-AP.  Computations establish that equality holds and so $\aw(\Z_{17},3)=4$.  }\end{ex}

\begin{defn}{When dealing with a coloring $c$ of $\Z_{st}$,  the  {\it $i^{\it th}$ residue class} modulo $s$ is $R_i:=\{j\in \Z_{st}:j\equiv i\pmod{s}\}$ and the {\it $i^{\it th}$ residue palette} modulo $s$ is $P_i:=\{c(\ell) :\ell\in R_i\}$. For  a positive integer $t$, % and    a coloring $c$ of $\Z_{2 t}$,
we call the elements of the two residue classes, $R_0$ and $R_1$, modulo $2$ in $\Z_{2t}$ the {\em even numbers} %, denoted by $E$,
and the {\em odd numbers}, respectively.%, denoted by $O$.
}
\end{defn}

%\subsubsection{Proof of Theorem~\ref{RM4}, $\aw(\Z_{2^m},3)$}

%\newpage
%------------------------------------------------------------------------------------------
\subsubsection{Consequences of results in Jungi\'c et al.}\label{ssec:newawZn3}

In this section we state two important results of  Jungi\'c et al. \cite{JLMNR03} and derive implications.  These are  used in  the proof of Theorem \ref{RM4}.  The next result is an equivalent form of Theorem 3.5 in that paper.

\begin{thm}\label{JLMNR3.5}{\rm \cite[Theorem 3.5]{JLMNR03}}
Let $n$ be a positive integer. Then $\aw(\Z_n, 3) = 3$ if and only if one of the following conditions is satisfied:

\bit
\item[a)] 
$n$ is a power of $2$,
\item[b)] 
$n$ is prime and $2$ is a generator of the multiplicative group $\Z_n^\times$,
\item [c)]
$n$ is prime, $\frac{n-1}{2}$ is odd, and the order of 2 in $\Z_n^\times$ is $\frac{n-1}{2}$.
\eit

\end{thm}

\begin{thm}\label{JLMNR3.2}{\rm \cite[Theorem 3.2]{JLMNR03}}
Let $n$ be an odd positive integer and $q$ be the smallest prime factor of $n$.  Then every $3$-coloring of $\Z_n$ in which every color class has at least $\frac n q+1$ elements contains a rainbow $3$-AP.
\end{thm}

A coloring $c$ of $\Z_n$ is an {\em extremal coloring} if $c$ is an exact $(\aw(\Z_n,3)-1)$-coloring  of $\Z_n$ with no rainbow $3$-AP.  A coloring $c$ of $\Z_n$ is a {\em singleton coloring} if  some color is used exactly once. %A coloring $c$ of $\Z_n$ is a {\em singleton extremal coloring} if $c$ is an exact $(\aw(\Z_n,3)-1)$-coloring  of $\Z_n$ with no rainbow $3$-AP such that  some color is used exactly once.

\begin{prop}\label{newsingleton}
Let $p$ be a prime positive integer. Then $3\le \aw(\Z_p,3)\le 4$, and   $\aw(\Z_p,3)= 4$ implies every extremal coloring of $\Z_p$ is a singleton coloring. 
   \end{prop}
\bpf First we suppose $\aw(\Z_p,3)\geq 5$ and let $c$ be an extremal coloring with $r=\aw(\Z_p,3)-1\geq 4$ colors. That is, $c$ does not have a rainbow $3$-AP. Hence, there is at least one color class with more than one element. We can define a 3-coloring $\hat c$  by partitioning the color classes of $c$ into three sets and defining the color classes of $\hat c$ to be the unions of the color classes in the sets.  Clearly $\hat c$ is a $3$-coloring of $\Z_p$ that does not have a rainbow $3$-AP. By Theorem~\ref{JLMNR3.2}, there exists a color class of $\hat c$ with $<\frac{p}{p}+1=2$ elements. This means that all but one of the color classes has a single element (and $r=4$). Without loss of generality, let the singleton colors be in positions $0, x$, and $y$, with $0<x<y<p$ (when viewed as integers rather than elements of $Z_p$).  In order to avoid the rainbow 3-AP in $c$ consisting of $0,\frac x 2, x$, $x$ must odd, and similarly $y$ must be odd as well.  But then $x, \frac {y-x}2 ,y$ is a a rainbow 3-AP in $c$, contradicting $\aw(\Z_p,3)\ge 5$.  

Next, suppose that $\aw(\Z_p,3)=4$ and let $c$ be an extremal coloring of $r=3$ colors. Since $c$ has no rainbow $3$-AP, Theorem~\ref{JLMNR3.2} gives that there is a color class with one element. That is, $c$ must be a singleton coloring.\epf

Since  $\aw(\Z_p,3)=3$ implies $\Z_{p}$ has the singleton extremal coloring $c(0)=1$ and $c(i)=2$ for every $i\not\equiv 0\pmod p$, the next corollary is immediate.
\begin{cor}   Every prime $p$ has a singleton extremal coloring of $\Z_p$. 
\end{cor} 

Since $\aw(\Z_{2^m},3)=3$, there are infinitely many values of $n$ for which $\aw(\Z_n,3)=3$. As stated in Theorem \ref{RM4},  $\aw(\Z_n,3)$ can be be made arbitrarily large and computed from the values of $\aw(\Z_p,3)$ for the  prime factors $p$ of $n$.    For  primes $p$, $\aw(\Z_p,3)>3$  seems rare from the data in Table \ref{tab:znk3}.  However, it follows from Theorem \ref{JLMNR3.5}  that there are infinitely many primes $p$ such that $\aw(\Z_p,3)=4$:
\begin{cor}
If $p$ is a prime and $p \equiv 1 \pmod 8$, then $\aw(\Z_p,3)=4$.  There are infinitely many such primes.
\end{cor}

\bpf
Since $p \equiv 1 \pmod 8$, $\frac{p-1}{2}$ is even. Also, $2$ must be a square in $\Z^\times_p$, which implies it is not a generator of $\Z^*_p$. So by Theorem \ref{JLMNR3.5}, $\aw(\Z_p,3) \neq 3$. Then by Proposition~\ref{newsingleton}, $\aw(\Z_p,3)=4$. By Dirichlet's Theorem there are infinitely many  primes $p \equiv 1 \pmod 8$.
\epf

%------------------------------------------------------------------------------------------
\subsubsection{Proof of Theorem~\ref{RM4}}

%As we see from the computed values in Table \ref{tab:znk3},  $\aw(\Z_n,3)$ does not appear to have an increasing lower bound. In fact,  $\aw(\Z_n,3)=3$ whenever $n$ is a power of 2, as proved in \cite[Theorem 3.5]{JLMNR03} (see Section \ref{ssec:newawZn3} for the statement of this result and further discussion).  

In this section  we present a series of results that lead to equivalent lower and upper bounds on $\aw(\Z_n,3)$ in terms of the prime factorization of $n$, establishing Theorem~\ref{RM4}. %, and so $\aw(\Z_n,3)$ is determined for arbitrarily large $n$ such that $\aw(\Z_p,3)$ is known for all prime factors $p$ of $n$.  This includes all $n$ such that all prime factors are less than 100.  

The next result gives our main recursive upper bound for $\aw(\Z_n,3)$.

\begin{prop}\label{upperpd}  Suppose $s$ is odd, and either $t$ is odd or $t=2^m$.  Then
	\[\aw(\Z_{st},3) \le  \aw(\Z_{t},3)+ \aw(\Z_s,3)  - 2.\]
\end{prop}

Proposition~\ref{upperpd} is established by Propositions \ref{upperpdoddt} ($t$ odd) and \ref{upperpd2mt} ($t=2^m$) below, after the proofs of necessary  preliminaries.  %, but we first point out an immediate corollary to Proposition~\ref{upperpd} together with Proposition~\ref{CH:thm:awpn3lowerbnd}.

%\begin{cor}\label{LH:awZ3power}
%	Suppose $p$ is an odd prime such that  $\aw(\Z_p,3) = 3$.  Then $\aw(\Z_{p^m},3) = m+2$.  \end{cor}

%Examples of primes $p$ to which Corollary \ref{LH:awZ3power} applies include 3, 5, 7, 11, 13, and 19; additional primes may be found in Table \ref{tab:znk3}.  Next we prove a technical lemma used in the proof of Proposition \ref{upperpdoddt} and elsewhere.

%%%%%%%%%%%%%%%
% If we get the theorem aw([3^m],3)=m+3 then comment on equality of aw([3^m],3) and aw(Z_{3^m},3) here
%%%%%%%%%%%%%%%

\begin{prop}\label{biga0}
Let $s$ be an odd positive integer. Suppose $c$ is a  coloring of $\Z_{st}$ that does not have a rainbow $3$-AP.  Let $R_0,R_1,\dots, R_{s-1}$  be the residue classes modulo $ s$ in $\Z_{st}$ with associated residue palettes $P_i $.  Let $m$ be an index such that $|P_m| \ge |P_i|$ for all $i$.  Then $|P_i\setminus P_m| \le 1$ for all $i$. \end{prop}

\bpf For arbitrary nonnegative integers $h$ and $j$, we show that  $|P_{h+j}\setminus P_h|\geq 2$ implies $P_h=P_{h+2j}$. Assume $|P_{h+j}\setminus P_h|\geq 2$.  Suppose first that $P_{h+2j}\setminus P_h$ is not empty and $z\in P_{h+2j}\setminus P_h$. Since $|P_{h+j}\setminus P_h|\geq 2$, we can pick some $y\in P_{h+j}\setminus P_h$ other than $z$. Let $\ell_y,\ell_z\in \mathbb{Z}_{st}$ with $\ell_y\in R_{h+j},\ell_z\in R_{h+2j}$ and $c(\ell_y)=y$, $c(\ell_z)=z$. Define $\ell_x\!:=\!2\ell_y-\ell_z\in R_h$, so $x\!:=c(\ell_x)$ is a color in $P_h$.  By the choice of $y$, $y\neq z$; $z\neq x$ since $z\in P_{h+2j}\setminus P_h$ and $x\in P_h$; $x\neq y$ since $y\in P_{h+j}\setminus P_h$ and $x\in P_h$. Thus $\ell_x,\ell_y,\ell_z$ is a rainbow $3$-AP, a contradiction. Therefore we conclude that $P_{h+2j}\subseteq P_h$. With this condition, we consider the case  $P_h\setminus P_{h+2j}$ is not empty.  Let $x\in P_h\setminus P_{h+2j}$. Similarly, it is possible to pick $y\in P_{h+j}\setminus P_h$. % other than $x$.
Let $\ell_x,\ell_y\in \mathbb{Z}_{st}$ with $\ell_x\in R_h$, $\ell_y\in R_{h+j}$, and $c(\ell_x)=x$, $c(\ell_y)=y$. Thus $\ell_z\!:=\!2\ell_y-\ell_x\in R_{h+2j}$ and so $z\!:=c(\ell_z)$ is a color in $P_{h+2j}$. Again,  $x\neq y$ by the choice of $y$; $x\neq z$ since $x\in P_h\setminus P_{h+2j}$ and $z\in P_{h+2j}$; $y\neq z$ since $y\in P_{h+j}\setminus P_h$ and $z\in P_{h+2j}\subseteq P_h$. Since we again have a contradiction, $P_h=P_{h+2j}$.

Next we show that $|P_{h+j}\setminus P_h|\geq 2$ implies $|P_h\setminus P_{h+j}|\leq 1$. Suppose $|P_{h+j}\setminus P_h|\geq 2$ and $|P_h\setminus P_{h+j}|\geq 2$, and then show this leads to a contradiction. By the result just established, $P_h=P_{h+2j}$. Since $|P_{h+2j}\setminus P_{h+j}|=|P_h\setminus P_{h+j}|\geq 2$, $P_{h+j}=P_{h+3j}$. Therefore $P_h=P_{h+qj}$ whenever $q$ is even and $P_{h+j}=P_{h+qj}$ whenever $q$ is odd. Since $s$ is odd, the order $d$ of $j$ in $\mathbb{Z}_s$ is also an odd number. That means $P_h=P_{h+dj}=P_{h+j}$, which is a contradiction.

Finally, since $|P_m|$ is chosen to be maximum, $|P_m\setminus P_j|\geq 2$ whenever $|P_j\setminus P_m|\geq 2$, which is impossible. Hence $|P_j\setminus P_m|\leq 1$.
\epf

%\noi{\em Proof of Proposition~\ref{upperpd}.}
\begin{prop}\label{upperpdoddt}  Suppose $s$ and $t$ are  both odd.  Then
	$\aw(\Z_{st},3) \le   \aw(\Z_{s},3) + \aw(\Z_t,3) - 2$.
\end{prop}
\bpf
 	Let $c$ be a  coloring of $\Z_{st}$  that does not have a rainbow $3$-AP.  Consider the residue classes and residue palettes modulo $s$ and without loss of generality assume $|P_0| \ge |P_i|$ for all $i$.
	We claim that
	\beq\label{NWclaim}\left|\bigcup_{i=0}^{s-1}P_i\right| \le (\aw(\Z_s,3) - 1) + (\aw(\Z_t,3) - 1) -1. \eeq
	  The proof is by contradiction.
	Assume that  \eqref{NWclaim} is false, i.e., assume
	\beq\label{LH311contra}\left|\bigcup_{i=0}^{s-1}P_i\right| \ge (\aw(\Z_s,3) - 1) + (\aw(\Z_t,3) - 1) \eeq
	 and define a coloring $\hat c$ of  $\Z_{s} = \{0,1,\dots,s-1\}$ in the following way: Let $\alpha$ be a color not in $\bigcup_{i=1}^{s-1} (P_i \setminus P_0)$ and define
	\[\hat c(i) = \left\{\begin{array}{cc} \alpha & \text{if } P_{i}\subseteq P_0,\\				
	  \text{the element of }P_{i}\setminus P_0 & \text{if } P_{i}\not\subseteq P_0.\end{array}\right. \]  Note that Proposition~\ref{biga0} implies that  the required element in $P_i\setminus P_0$ is unique, so this coloring is well-defined.  Since $c$ does not have a rainbow $3$-AP, we know $|P_0| \le \aw(\Z_t,3) - 1$ so
	
	\[\left|\dsp\bigcup_{i=1}^{s-1} (P_i \setminus P_0)\right| \ge \left|\bigcup_{i=0}^{s-1} P_i \right| - (\aw(\Z_t,3) - 1) \ge (\aw(\Z_s,3) - 1) + (\aw(\Z_t,3) - 1) - (\aw(\Z_t,3) - 1) = \aw(\Z_s,3) - 1.\]
	Note that every color that is not in $P_0$, together with $\alpha$, is used in $\hat c$, so $\hat c$ uses at least $\aw(\Z_s,3) $ colors.  Thus a rainbow 3-AP exists in $\hat c$.
	
	We show that a rainbow $3$-AP in $\hat c$ implies a rainbow $3$-AP in $c$, providing a contradiction and establishing that \eqref{NWclaim} is true.
	Let $x,y,z$ be a rainbow $3$-AP in $\Z_s$ using coloring $\hat c$, with $ y = x+ d \pmod{s}$ and $z = x + 2d \pmod{s}$.  Since $x,y,z$ is rainbow, $\hat c(u) \neq \hat c(v)$ for all distinct  $u,v\in\{x,y,z\}$, and so at most one $u\in\{x,y,z\}$ has $\hat c(u)=\alpha$.
	Note that by definition $\hat  c(u) \in P_u$ or $\hat c(u)=\alpha$ for $u\in\{x,y,z\}$.
%	By  Proposition~\ref{biga0} and the choice of $\alpha$, $\hat c(u),\hat c(v)\notin P_{w}$ for $\{u,v,w\}=\{x,y,z\}$.

\begin{mycases}
\mycase{$\hat  c(z) \ne \alpha$ and $\hat  c(y) \ne \alpha$.\label{case:LH1}}	
	 Then we can find $g_2$ and $g_3$ such that $c(g_2s + y) = \hat c(y)$ and $c(g_3s+z) = \hat c(z)$.  Define $d' := (g_3s + z) - (g_2s + y) $.  Then
\begin{align*}
 (g_3s + z) -d' = (g_2s + y) &\equiv y \pmod{s} \\
	  (g_3s + z) -2d' = 2g_2s+2y - g_3s - z  \equiv  2y - z \equiv  2(x + d) - (x +2d) &\equiv x \pmod{s}.
\end{align*}
With $\ell\!:=(g_3s + z) -2d'$, consider the 3-AP  $\{\ell,(g_3s + z) -d',(g_3s + z)\}$.  We show that this 3-AP is rainbow:  Note that $\hat c(y)\notin P_0$ and $\hat c(z)\notin P_0$.  If $\hat c(x)=\alpha$, then $P_x\subseteq P_0$, so  $\ell\in  R_x$ implies $c(\ell) \neq \hat c(y)=c(g_2s + y)$ and $ c(\ell) \neq \hat c(z)=c(g_3s + z)$.  If $\hat c(x)\ne \alpha$, then $\hat c(x)$ is the unique element of $P_x\setminus P_0$ and $\hat c(x) \neq \hat c(y), \hat c(z)$, so  $\ell\in  R_x$ implies $c(\ell) \neq \hat c(y)$ and $c(\ell) \neq \hat c(z)$.
	 Thus $c$ has a rainbow $3$-AP, contradicting our assumption \eqref{LH311contra}.  The case where both $\hat  c(x) \ne \alpha$ and $\hat  c(y) \ne \alpha$ is symmetric to Case~\ref{case:LH1}. So only Case~\ref{case:LH3} remains.
	
% \mycase{$\hat  c(x) \ne \alpha$ and $\hat  c(y) \ne \alpha$.} This case  is symmetric with Case \ref{case:LH1}.
	
\mycase{$\hat c(y)=\alpha$.\label{case:LH3}}	
	  Then $\hat c(x)\in P_x\setminus P_0$ and $ \hat c(z)\in P_z\setminus P_0$, so we can find $g_1$ and $g_3$ such that $c(g_1s + x) = \hat c(x)$ and $c(g_3s+z) = \hat c(z)$, and define  $e := (g_3s + z) - (g_1s + x) $.  Since $st$ is odd, 2 is invertible in $Z_{st}$ and there exists $d'$ such that $2d'\equiv e\pmod{st}$, and hence $2d'\equiv e \pmod{s}$. Also, $e\equiv z-x\equiv 2d \pmod{s}$.  Thus $2d\equiv 2d'\pmod{s}$ and so $d\equiv d'\pmod{s}$ since $s$ is odd.
	 Then
\begin{align*}
(g_1s + x) +2d' \equiv (g_1s + x)+((g_3s + z) - (g_1s + x))=g_3s+z&\equiv z \pmod{s} \\
	  (g_1s + x) +d' \equiv x+d &\equiv y \pmod{s}.
\end{align*}
With $\ell:=(g_1s + x) +d'$, the 3-AP  $\{(g_1s + x) ,\ell,(g_1s + x)+2d'\}$ is rainbow, because $\ell\in R_y$ and $P_y\subseteq P_0$, so $ c(\ell)\ne \hat c(x)=c(g_1s + x)$ and $ c(\ell)\ne \hat c(z)=c(g_3s + z)$.
\end{mycases}
	 In all cases, $c$ has a rainbow $3$-AP, contradicting our assumption \eqref{LH311contra}.	%\hfill $\Box$
\epf	
	%\medskip
	
Next we prove two technical propositions used in the proof of Proposition \ref{upperpd2mt}, Propositions~\ref{two} and~\ref{MYDSone}.
\begin{prop}\label{two}
Let $m$ and $s$ be positive integers with  $s$ odd.
Suppose $c$ is a coloring of $\Z_{2^m s}$ using at least $r:=\aw(\Z_s, 3) + 1$ colors that does not have a rainbow $3$-AP.
Let $R_0,R_1,\dots, R_{s-1}$  be the residue classes modulo $ s$ in $\Z_{2^m s}$ with associated residue palettes $P_i $.  Then $1\le  |P_i| \leq 2$ for $i=0,\dots, s-1$, and all palettes $P_i$ of size two share a common color.
\end{prop}

\begin{proof}
Since $P_i$ is nonempty, $1 \leq |P_i|$.
Observe that the coloring  $c$ of $R_i$ induces a coloring on $\Z_{2^m}$ that uses only the colors in $P_i$ and cannot contain a rainbow 3-AP.  Thus $ |P_i| \leq 2$ by Theorem \ref{JLMNR3.5}, establishing the first statement.

By Proposition~\ref{biga0}, each pair of residue palettes  of size two must intersect.
Suppose the %residue color
palettes of size two do not all intersect in a common color.  Then there are exactly three colors $\alpha, \beta, \gamma$ that are used  by all the %residue color
palettes of size two, and there are exactly three distinct palettes of size two, each consisting of two of these three colors.  We show this configuration leads to a contradiction.
%Let  be the three colors used by palettes  of size two. %, and let $x_1,\dots, x_q$ be the colors used by palettes  of size one.

Create a coloring $\hat c$ of $\Z_s$ by the following method:
\[\hat c(i) = \left\{\begin{array}{cl} c(i) & \mbox{if } |P_{i}|=1,\\				
	  \beta & \mbox{if } P_i=\{\alpha,\beta\},\\%|P_i|=2\mbox{ and }h\not\in P_{i}
	  	  \text{the unique element of }P_{i}\setminus \{\gamma\} & \mbox{if } |P_i|=2\mbox{ and }\gamma \in P_{i}.
	  \end{array}\right. \]
Observe that $\hat c$ uses  $r$ colors if there exists $i$ such that $P_i=\{\gamma\}$ and $r-1=\aw(\Z_s,3)$ colors otherwise, so in either case $\hat c$ must have a rainbow 3-AP.
	  Suppose that $\{x, y, z\}$ is a rainbow 3-AP for the coloring $\hat c$ of $\Z_s$.
Since $\hat c(x)$, $\hat c(y)$, and $\hat c(z)$ are distinct colors, at least one of the palettes $P_x, P_y, P_z$ contains only one color.
Consider the sizes of $P_{x}$, $P_{y}$, and $P_{z}$.  %The case $|P_{x}| = |P_{y}| =|P_{z}| = 1$  immediately yields a rainbow 3-AP in the original coloring $c$  of $\Z_{2^ms}$, which is a contradiction.

\begin{mycases}

\mycase{$|P_{z}| = 1$.}  Observe that $\hat{c}(i)$ is always an element in $P_i$ by our definition of $\hat{c}(i)$. Pick $n_1\in R_{x}$ and $n_2\in R_{y}$ such that $c(n_1)=\hat{c}(x)$ and $c(n_2)=\hat{c}(y)$. Thus $n_3:=2n_2-n_1$ is an element in $R_{z}$ and so $c(n_3)=\hat{c}(z)$. Since $\hat{c}(x),\hat{c}(y),\hat{c}(z)$ are all distinct, $\{n_1,n_2,n_3\}$ is a rainbow 3-AP.  The case  $|P_{x}| = 1$ is symmetric.

\mycase{$|P_{x}| = |P_{z}| = 2$ and $|P_{y}| = 1$.}
Since $\hat c(x) \neq \hat c(z)$, it must be that $\{ \hat c(x), \hat c(z) \} = \{ \alpha, \beta\}$.
Without loss of generality, we assume that $\hat c(x) = \beta$ and $\hat c(z) = \alpha$.
By the definition of $\hat c$, $P_{z} = \{\alpha,\gamma\}$.
Then $P_{x}$ is one of $\{\alpha,\beta\}$ or $\{\beta,\gamma\}$.
 If  $\hat c(y)\not\in P_{x} \cup P_{z}$, then any 3-AP $\{n_1,n_2,n_3\}$ where $n_1 \in R_{x}$ and $c(n_1)=\beta$, $n_2\in R_{y}$, and $n_3\in R_{z}$ is a rainbow 3-AP in the original coloring.
Thus, $\hat c(y) \in P_{x}\cup P_{z}\subseteq \{\alpha,\beta, \gamma\}$, but $\hat c(y)\notin \{\alpha,\beta\}=\{ \hat c(x), \hat c(z) \}$, so $\hat c(y) = \gamma$. % and $P_{x} = \{\gamma,\beta\}$.
Note that this implies $\hat c$ uses all $r$ colors.

Since this is the final case, and all previous cases led to contradictions, every rainbow 3-AP in $\Z_s$ given by the coloring $\hat c$ must be of the form $\{x,y,z\}$ where $\{\hat c(x), \hat c(z)\} = \{\alpha, \beta\}$ and $\hat c(y) = \gamma$.
Create a new coloring $c'$ of $\Z_s$ where $c'(i) = \begin{cases} c(i) & \text{ if }\hat c(i) \neq \gamma,\\ \beta & \text{ if }\hat c(i) = \gamma.\end{cases}$

Now, every 3-AP that was previously non-rainbow in $\hat c$ remains non-rainbow in $c'$ and the rainbow 3-APs (which necessarily used the colors $\alpha$, $\beta$, and $\gamma$) are no longer rainbow.
Thus, this coloring $c'$ does not have a rainbow 3-AP,
but $c'$ uses  $r-1=\aw(\Z_s,3)$ colors, a contradiction.
\end{mycases}
The above cases show that  having no common color among the  palettes of size two leads to a contradiction.
Therefore,  all of the residue  palettes of size two share a common color.
\end{proof}

\begin{prop}\label{MYDSone}
Suppose $c$ is a coloring of $\Z_{2 t}$ ($t\ge 1$)  that does not have a rainbow $3$-AP.  Let $A$ and $B$ denote the residue palettes modulo $2$ in $\Z_{2t}$ associated with the even and odd numbers, respectively. Then $|A\setminus B|\le 1$ and $|B\setminus A|\le 1$.
\end{prop}
\bpf
It suffices to show that $|A\setminus B|\le 1$ for every such coloring $c$ because if $|B\setminus A|\geq 2$, then the coloring defined by the rotation $c'(x):=c(x+1)$ has the roles of $A$ and $B$ reversed. Suppose not, so there exist two colors $\alpha, \gamma$ that appear only in $A$.
Let $n_1=2m_1$ and $n_3=2m_3$ be even elements such that $c(n_1)=\alpha$ and $c(n_3)=\gamma$.
We can select $m_1$ and $m_3$ such that $0 \leq m_1 < m_3 < t$. Performing arithmetic in the integers, we can choose $m_3-m_1$ to be minimum with respect to the fact that the set of colors  $\{c(2m_1), c(2m_3)\}$ is $\{\alpha,\gamma\}$.
Let $n_2 = m_1+m_3$ and observe that $\{n_1, n_2, n_3\}$ is a 3-AP and hence is not rainbow.
Therefore, $n_2$ must have the color $\alpha$ or $\gamma$ and thus is even.
However, this implies that $n_2 = 2m_2$ and $m_1 < m_2 < m_3$, while one of the sets of colors  $\{c(2m_1),c(2m_2)\}$ or $\{c(2m_2),c(2m_3)\}$ is  $\{\alpha,\gamma\}$, so one of the pairs $(m_1,m_2)$, $(m_2,m_3)$ violates our extremal choice.
\epf

\begin{prop}\label{upperpd2mt}
Let $m$ and $s$ be positive integers with  $s$ odd. Then
\[
	\aw(\Z_{2^m s},3) \le \aw(\Z_s, 3) +1.
\]
\end{prop}

\begin{proof} The result is immediate for $s=1$ because Theorem~\ref{JLMNR3.5} gives that $\aw(\Z_{2^m},3)=3$ and because $\aw(\Z_s,3)=s+1$ for $s<3$, so assume $s\ge 3$. We proceed by induction on $m$.
Suppose $c$ is an exact $r$-coloring of  $\Z_{2^m s}$ with $r = \aw(\Z_s, 3) +1$ that does not have a rainbow 3-AP.   Let $A$ and $B$ denote the residue palettes of  the even and odd numbers, respectively.  By Proposition~\ref{MYDSone}, $|A\setminus B|\le 1$ and $|B\setminus A|\le 1$, so $|B|\ge r-1$ and $|A|\ge r-1$.
  The base case $m=1$ is then immediate, because   the coloring of the even  numbers of $\Z_{2s}$ induces a coloring of $\Z_{s}$, so a rainbow 3-AP necessarily exists, producing a contradiction.

Now consider $m>1$.  As usual $R_i, i=0,\dots,s-1$, are the residue classes modulo $s$ of $\Z_{2^m s}$ and $P_i, i=0,\dots,s-1$, are the residue  palettes. Recall that by Proposition~\ref{two}, $1 \le |P_i| \le 2$ for all $i$.   For $0 \le i \le s-1$,
let $A_i=P_i\cap A$ be the colors appearing on the even numbers in $R_i$, and let $B_i=P_i\cap B$ be the colors appearing on the odd numbers in $R_i$.
Thus, $P_i = A_i \cup B_i$, $A=\bigcup_{i=0}^{s-1} A_i$, and $B=\bigcup_{i=0}^{s-1} B_i$.
We claim that  $|A| = |B| =r-1$.
To see this,
observe that the even elements induce a coloring of $\Z_{2^{m-1} s}$, so if $|A| = r$, then a rainbow 3-AP necessarily exists, since $r \geq \aw(\Z_{2^{m-1}s},3) $ by the induction hypothesis.  Thus $|A| \le r-1$, and so $|A|= r-1$.  The proof that $|B|=r-1$ is analogous.

Since $|A| = |B| =r-1$, there exist colors $\alpha,\beta$ such that $A \setminus B = \{\alpha\}$ and $B\setminus A = \{\beta\}$.
Assume $\alpha \in A_u$ and $\beta \in B_v$.
Let $j = v-u$, hence $\beta \in B_{u+j} = B_v$.  Since there is no rainbow 3-AP, $u+2j$ must have a color in palette $A$, $\alpha \in A_{u+2j}$, which then implies $\beta \in B_{u+3j}=B_{v+2j}$.
Iterating this process gives that $\alpha \in A_{u+2\ell j}$ and $\beta \in B_{v+2\ell j}$ for all $\ell \geq 0$.
Since $s$ is odd, we have that for all $q\ge 0$, $A_{u+q j }$ is of the form $A_{u+2\ell j }$ for some $\ell$ and similarly, every $B_{u+qj}=B_{v+(q-1) j}$ is of the form $B_{v+2\ell j }$ for some $\ell$.
Therefore, $P_{u+q j } = \{\alpha,\beta\}$ for all $q \geq 0$.
By Proposition~\ref{two}, there is a common color for palettes of size two, and thus one of $\alpha$ or $\beta$ is this common color.
Without loss of generality, assume that $\alpha$ is the common color for all  palettes of size 2.
This implies that $|B_i| =1$ for all $0 \le i \le s-1$.
Hence, defining $\hat c(i)$ to be the unique color in $B_i$ defines an exact $(r-1)$-coloring of $\Z_s$ that avoids rainbow 3-APs.
However, $r - 1 = \aw(\Z_s, 3)$, a contradiction.
\end{proof}

 %Our  results give more than just certain prime powers such as $\aw(\Z_{3^m},3)$.

%Lemma~\ref{LH:Zn3primefactor} below follows from \cite[Theorem 3.5]{JLMNR03}, Propositions~\ref{CH:thm:awpn3lowerbnd} and Proposition~\ref{upperpd} by induction.

%\begin{lem}\label{LH:Zn3primefactor}  For any integer $n\ge 3$,	\[2+f_2+  \sum_{a=3}^{\mv}f_a\le \aw(\Z_n,3) \le 2+f_2+  \sum_{a=3}^{\mv}(a-2)f_a.\] \end{lem}

Proposition~\ref{upperpd} is now established from Proposition~\ref{upperpdoddt} and Proposition~\ref{upperpd2mt}.  We now turn our attention to establishing the lower bound.  

	\begin{prop}\label{NW:singleton}
		Suppose $s$ is odd and $\Z_s$ has  a singleton extremal coloring.  Then for $t\ge 2$, \[ \aw(\Z_{st},3)\ge \aw(\Z_t,3) + \aw(\Z_s,3) - 2 .\]
	\end{prop}
	\bpf Let $c_s$ be a singleton extremal coloring of $\Z_s$.
		Note that we can shift $c_s$ so that $c_s(0)$ is the color that is used exactly once.  Choose a coloring $c_t$ of $\Z_t$ using $\aw(\Z_t,3)-1$ colors not used by $c_s$ that does not have a rainbow 3-AP.  Let $R_0,R_1,\dots,R_{s-1}$ be the residue classes modulo $s$ in $\Z_{st}$.  Define a coloring $\hat c$ of $\Z_{st}$ as follows:  For $i=1,\dots,s-1$ and $\ell \in R_i$, $\hat c(\ell):=c_s(i)$, and for $0\le j\le t-1$,  $\hat c(js):= c_t(j)$.
				Notice that we now have an exact $\aw(\Z_s,3) - 2 + \aw(\Z_t,3) - 1$ coloring of $\Z_{st}$ because we have removed color $c_s(0)$.  Clearly, if a 3-AP is within some residue class it is not rainbow.  Because $s$ is odd, $d\not\equiv 0\pmod s$ implies $2d\not\equiv 0\!\pmod s$ and $2d\not\equiv d\!\pmod s$, so a 3-AP that is not entirely within one residue class has elements in three different residue classes.  But a rainbow 3-AP with elements in three different residue classes would imply a rainbow 3-AP in $c_s$,  which does not exist.  So we have found a coloring of $\Z_{st}$ using $\aw(\Z_t,3) + \aw(\Z_s,3) - 3$ colors that does not have a rainbow 3-AP.  Thus $  \aw(\Z_{st},3)\ge \aw(\Z_t,3) + \aw(\Z_s,3) - 2$.
	\epf

\begin{cor}\label{LH:Zn3primefactor3} For an  integer $n\ge 2$, 	\[ \aw(\Z_n,3) = 2+f_2(n)+  f_3(n) +2f_4(n).\]
\end{cor}
\bpf  By Proposition~\ref{newsingleton}, every odd prime factor $p$ has $3\le \aw(\Z_p,3)\le 4$.  Apply Proposition \ref{upperpd}, removing one odd prime $s$ at a time and observing that for $\aw(\Z_s,3)=3$, $\aw(\Z_s,3)-2$ adds one to the total, whereas for $\aw(\Z_s,3)=4$, $\aw(\Z_s,3)-2$ adds two to the total.  Thus $\aw(\Z_n,3)\le 2+f_2(n)+  f_3(n) +2f_4(n)$.  For the reverse inequality, suppose $p$ is an odd prime. Then  every extremal  coloring of $\Z_p$ is a singleton coloring by Proposition~\ref{newsingleton}.  So we can apply Proposition~\ref{NW:singleton} to remove one odd prime at a time to show that $\aw(\Z_n,3)=2+f_2(n)+f_3(n)+2f_4(n)$.  \epf

\begin{rem}\label{MYrem}  The constructive proof of Proposition~\ref{NW:singleton} gives a singleton extremal coloring of $\Z_n$ from the singleton extremal colorings of the prime factors of $n$.  Since $\Z_{2^m}$ has the singleton extremal coloring $c(0)=1$ and $c(i)=2$ for every $i\not\equiv 0\pmod {2^m}$, every positive integer has a singleton extremal coloring.
\end{rem}

\begin{prop}\label{LH:Zn3primeval}    For all primes $p< 100$,
	$\aw(\Z_p,3) = 3$ if $p\notin Q_4:=\{17, 31, 41, 43, 73, 89,  97\}$ and $\aw(\Z_p,3) = 4$ if $p\in Q_4$.  %Furthermore, for every prime $p<100$, $\Z_p$ has a singleton extremal coloring.
\end{prop}

\bpf  The statement that for any prime $p< 100$,
	$\aw(\Z_p,3) = 3$ if $p\notin Q_4$ and $\aw(\Z_n,3) = 4$ if $p\in Q_4$ has been verified computationally (see Table \ref{tab:znk3}). %If $\aw(\Z_p,3) = 3$, then $c(0)=1, c(i)=2$ for $i>0$ is a singleton extremal coloring of $\Z_p$.  For each $p\in Q_4$, a singleton extremal coloring of $\Z_p$  is given in Table \ref{LHtab:singleton}.
\epf

The next example illustrates the use of Corollary~\ref{LH:Zn3primefactor3} %Theorem \ref{RM4} 
to compute $\aw(\Z_n,3)$ in the case that 
 every prime factor  of $n$ is less than $100$.
% {\red[Should this be earlier?]}

\begin{ex}\label{LH:newex}{Let $n=14,582,937,583,067,568$. Since $n= 2^4 \cdot 3\cdot 11^2\cdot 13\cdot 17^2\cdot 53^3\cdot 67^2$, $\aw(\Z_n,3) = 3+ 9+2\cdot 2=16$.}\end{ex}

%\begin{ex}\label{LH:newex}{\rm  Let $n=3,995,513,883,250,032$. Since $n= 2^4 \cdot 3\cdot 11\cdot 13^2\cdot 53^3\cdot 67^3$, 	$\aw(\Z_n,3) = 3+10=13$ because 2 is a factor and  there are 10 prime factors, all with $\aw(\Z_{p},3)=3$ (see Table \ref{tab:znk3}).}\end{ex}

%It now becomes valuable to determine the value of $\aw(\Z_{p},3)$ for $p$ prime, and whether a coloring satisfying the hypotheses of Lemma~\ref{NW:singleton} exists.    The following has been established by computation.

%------------------------------------------------------------------------------------------
%\pagebreak
\subsection{Main results for $\aw(\Z_n,k), k\ge 4$}\label{ssec:awZnk}

In this section, we specialize to the case where $k\geq 4$ and prove Theorem \ref{RM5}.
%Next, we establish bounds for $\aw(\Z_n,k)$ which are similar to those found for $\aw([n],k)$ in Section \ref{ssec:awnk}.
Corollary \ref{KH:Thm:littleohub} below, which follows from Corollary \ref{CEthm:upawboundone} and  Remark \ref{KH:Thm:ZnVSBracketn},  gives us $ne^{-\log\log\log n - \omega(1)}$ as an upper bound for $\aw(\Z_n,k)$.

\begin{cor}\label{KH:Thm:littleohub}
    For every fixed positive integer $k$, $\aw(\Z_n,k)=o\left(\frac{n}{\log\log n}\right)$.
\end{cor}

Our lower bound for $\aw(\Z_n,k)$ when $n>12$ is presented in Lemma~\ref{KH:Thm:BehrendtypelbZn}.

\begin{lem}\label{KH:Thm:BehrendtypelbZn}
    There exists an absolute constant $b>0$ such that for all $c>3$, $\frac{n}{c}\geq4$ and $k\geq 4$,
    \[
        \aw(\Z_n,k) > \left(\frac{n}{c}\right)e^{-b\sqrt{\log\left(n/c\right)}} =ne^{-b \sqrt{\log(n/c)} - \log c} = n^{1-o(1)}.
    \]
\end{lem}

Lemma~\ref{KH:Thm:BehrendtypelbZn} is proven using the Behrend construction from Section \ref{ssec:awnk} and using Proposition~\ref{KH:Lem:punc4APZn} below.  The Behrend construction in the integers $\{1,\ldots,m\}$ has no punctured 4-AP and size $me^{-b\sqrt{\log m}}$ for some absolute constant $b$.
%A similar lower bound can be found when $k\geq 6$ if we rainbow color a subset of $\Z_n$ which has no $\left\lfloor\frac{k}{2}\right\rfloor$-APs and color all remaining members of $\Z_n$ with a new color.  We can then apply Proposition $1.3$ from \cite{H05}, which gives a lower bound for the size of a maximal subset of $\Z_n$ with no non-constant $\left\lfloor\frac{k}{2}\right\rfloor$-APs.

\begin{prop}\label{KH:Lem:punc4APZn}
    Let $c>3$ be a real number, and let $\left[ \frac{n}{c}\right]$ denote the first $\lfloor \frac{n}{c}\rfloor$ consecutive residues in $\Z_n$.  Suppose $S\subseteq \left[ \frac{n}{c}\right]$ does not contain any punctured $4$-APs.  Then $\aw(\Z_n,k) > |S|+1$ for all $k\geq 4$.
\end{prop}
\bpf
    Color each member of $S$ a distinct color, and color each member of $\Z_n \setminus S$ with a new color called \emph{zero}. Each $i\in\Z_n$ with $\frac{n}{c} < i < n$ will be colored zero.  If $K = \{a_1,a_2,a_3,a_4\}$ is a rainbow $4$-AP in $\Z_n$, then at most one element of $K$ is not in $S$.  Without a loss of generality, assume $K$ is ordered as $a_1, a_2, a_3, a_4$ and  $a_3,a_4 \in S$.  Then there exists $d\in\Z$ such that  $d \equiv a_4-a_3 \pmod n$  and $|d|\leq \frac{n}{c}$.

    Suppose $a_2\in S$.  Because $|d|\leq \frac{n}{c}<\frac{n}{2}$, we must have that $a_2,a_3,a_4$ is a $3$-AP in $\left[ \frac{n}{c}\right]$.  This contradicts the fact that $S$ contains no punctured $4$-APs, so we must have $a_2 \not\in S$ and $a_1\in S$.  However, since $2|d|\leq \frac{2n}{c}<\frac{(c-1)n}{c}$, we must have that $a_1,a_3,a_4$ is a punctured $4$-AP in $\left[ \frac{n}{c}\right]$.  This is a contradiction, so $a_1\not\in S$.

    This means that $K$ could not have been rainbow, so we have a $(|S|+1)$-coloring of $\Z_n$ with no rainbow $4$-APs.
\epf

We use the bound for the Behrend construction in Lemma~\ref{CEthm:Behrendbound} to obtain the bounds for $\aw(\Z_n,k)$, $k\geq 4$:
\[
    ne^{-b \sqrt{\log(n/c)} - \log c}< \aw(\Z_n,k) \leq ne^{-\log\log\log n - \omega(1)}.
\]
This completes the proof of Theorem \ref{RM5}.

%------------------------------------------------------------------------------------------
%\pagebreak
\subsection{Additional results for $\aw(\Z_n,k)$}\label{ssec:awZnkmore}

%\subsection{Elementary results for $\aw(\Z_n,k)$}\label{ssec:awZnall}

 In this section, we present computed data for $\aw(\Z_n,k), k\ge 4$, establish the value of $\aw(\Z_n,k)$ for $k=n$, $n-1$, and $n-2$, and present some  examples that show some additional results fail to extend from $[n]$ to $\Z_n$.
 Table \ref{KH:Table:awznkcomputed} below lists the computed values of $\aw(\Z_n,k)$ for $k=4,\dots,n$ in the row labeled $n$.

\begin{table}[h!]
\centering
{\small
\begin{tabular}[h]{c |c c c c c c c c c c c c c c c c c c c}
$n\setminus k$ &  4 & 5 & 6 & 7 & 8 & 9 & 10 & 11 & 12 & 13 & 14 & 15 & 16 & 17 & 18 & 19\\
\hline
4& 4\\
5& 4& 5\\
6& 5& 5& 6\\
7& 4& 5& 6& 7\\
8& 6& 6& 7& 7& 8\\
9& 5& 6& 8& 8& 8& 9\\
10& 6& 8& 8& 8& 9& 9& 10\\
11& 5&  6&  7&  8&  9& 9& 10&11\\
12&8  &9 &10 &10 &11 &11 &11 &11&12\\
13&5 & 7&  8& 9& 10& 10& 11& 11& 12&13\\
14&6  &8 &10 &12 &12 &12 &12 &12 &13 &13&14 \\
15&8 &11& 12& 12& 12& 13& 14& 14& 14& 14& 14&15 \\
16& 8 &10 &10 &11 &14 &14 &14 &14 &15 &15 &15 &15&16 \\
17&6 & 8 &10& 11& 12& 12& 13& 14& 14& 15& 15& 15& 16&17 \\
18&8 &10 &13 &14 &14 &16 &16 &16 &17 &17 &17 &17 &17 &17&18 \\
19&6 & 9 &10 & 12 & 12 & 14 & 14 & 15 & 16 & 16 & 16 & 17 & 17 & 17 & 18 & 19
\end{tabular}
}
\caption{\label{KH:Table:awznkcomputed}Computed values of $\aw(\Z_n,k)$ for $k\geq 4$.}
\end{table}

Next we examine $\aw(\Z_n,k)$ for $k$ close to $n$.

\begin{prop} \label{SB1} For positive $n$ and $k$ we have $\aw(\Z_n,k)=n$ if and
only if $k=n$.
\end{prop}

\bpf  If $k=n$ the result is obvious.  Now suppose that $k<n$ and
consider an exact $(n-1)$-coloring of $\Z_n$.  Then there are two numbers with
the same color and all other numbers are colored distinctly.
Suppose $x$ and $y$ are the the two numbers with the same color.
Then $\{x+1,...,x+k\}$ is a $k$-AP that does not contain $x$, and so
is rainbow.  Therefore $\aw(\Z_n,k)\le n-1$.  \epf\vspace{-10pt}

\begin{cor} \label{SB2}   For positive $n$, $\aw(\Z_n,n-1)=n-1$.
\end{cor}

A pattern can be observed in the values of $\aw(\Z_n,n-2)$, and this is established in Proposition~\ref{SB3}.

\begin{prop}  \label{SB3} %For a prime $p\ge 5$, $\aw(\Z_p,p-2)=p-2$.
For positive $n\ge5$, if $n$ is prime then $\aw(\Z_n,n-2)=n-2$; otherwise $\aw(\Z_n,n-2)=n-1$.
\end{prop}

\bpf  We trivially
have a lower bound of $n-2$ for $\aw(\Z_n,n-2)$.  First we assume $n$ is prime.  We claim that  for
any two distinct elements $x$ and $y$ there is an $(n-2)$-AP that
misses $x$ and $y$.  To see this, simply form the $n$-AP with
$a=x$ and $d=(y-x)$, this will cover all of $\Z_n$ and now removing the
first two terms leaves us with an $(n-2)$-AP that
does not contain $x$ or $y$.  So suppose we have an exact $(n-2)$-coloring. Then
either there is one color that occurs three times or two colors that each
occur twice, and in either case all other colors occur exactly once.  In either case we can choose two numbers to avoid and
then the remaining $n-2$ numbers are rainbow, but as just noted above
the remaining $n-2$ numbers are an arithmetic progression.
Therefore every $(n-2)$-coloring contains a rainbow progression.

 When $n$ is not a prime, let $p$ be the smallest prime divisor of $n$ and consider the $(n-2)$-coloring formed by coloring $0$, $p$ and $2p$ monochromatically, with the remaining numbers  all given distinct colors.  This is an $(n-2)$-coloring (since $2p<n$ by assumption that $n\ge 5$).  We claim this coloring has no rainbow $(n-2)$-AP  (along with the upper bound of $n-1$, this claim  establishes the result). Suppose that $K=\{a,a+d,\ldots,a+(n-3)d\}$ is a rainbow  $(n-2)$-AP, so all the elements of $K$ are distinct and $K$ necessarily misses two of $0, p, 2p$. Since $\Z_n$ cannot have a proper subgroup of order $n-2$, extending $K$ to a $n$-AP necessarily produces all elements of $\Z_n$ and thus  $\{a+(n-2)d,a+(n-1)d\}\subseteq\{0,p,2p\}$.  But then we have that $p$ divides $d=((a+(n-1)d)-(a+(n-2)d))$, showing that this arithmetic progression can have at most $\frac n p<n-2$ terms, which is a contradiction.  \epf

Proposition \ref{SB1}  shows that the ``if'' direction of Theorem \ref{LHawnkequalsnthm} ($k\geq \lceil\frac{n}{2}\rceil + 1$ implies $\aw([n],k)=n$) does not extend to $\Z_n$.
%Some of the extensions from $[n]$ to $\Z_n$ were done in Section \ref{ssec:awZnk}.  Some additional extensions  are immediate; in others, computed data (given in Tables \ref{tab:znk3} and  \ref{KH:Table:awznkcomputed}) quickly shows that no extension is possible.
Example \ref{KH:Ex:Ext4} below shows that the extension of  Proposition \ref{thm:awmonepone} to $\Z_n$, which would assert that $\aw(\Z_n,k) \le  \aw(\Z_{n-1},k) +1$, is not true in general.  There are counterexamples in both the cases $k=3$  and $k\geq 4$.

\begin{ex}\label{KH:Ex:Ext4}{By Corollary \ref{LH:Zn3primefactor3}, $\aw(\Z_{30},3)=5$,  and $\aw(\Z_{29},3)=3$ (see Table \ref{tab:znk3} in Section \ref{ssec:awZn3}).  Furthermore,   $\aw(\Z_{8},4)=6$ and $\aw(\Z_{7},4)=4$ (see Table \ref{KH:Table:awznkcomputed}).
}\end{ex}

Example \ref{KH:Ex:Ext7} below shows that Theorem \ref{thm:awnsumoverns}, which bounds the anti-van der Waerden number of a sum in terms of the anti-van der Waerden numbers of the summands, does not extend to $\Z_n$.

\begin{ex}\label{KH:Ex:Ext7}{According to our computed data (see Table \ref{KH:Table:awznkcomputed}),
	\[
		\aw(\Z_{12},4) = 8 > 4 + 4 - 1 = \aw(\Z_5,4) + \aw(\Z_7,4) - 1.
	\]
There are also examples for $k=3$, such as  $ \aw(\Z_{54},3)= 6>3+3-1=\aw(\Z_{47},3)+\aw(\Z_{7},3) - 1 $.  %So it doesn't extend for $k\geq 3$.
}\end{ex}

%%%%%%%%%%%%%%%%%%%%%%%%%%%%%%%%%%%%%%%%%%%%%%%%%%%%

\section{Computation}\label{sec:compute}

Many of the results we have proved in this paper were first conjectured from examination of data.
In this section, we briefly discuss an efficient algorithm to find an exact $r$-coloring of $[n]$ or $\Z_n$ that avoids a rainbow $k$-AP, if such a coloring exists.
For the sake of brevity, we will focus on the case of coloring $[n]$ since this case has a few extra properties that the $\Z_n$ case does not.
Specifically, we have $[m] \subseteq [n]$ for all $m \leq n$ while $\Z_n$ contains a copy of $\Z_m$ if and only if $m$ divides $n$.
%In these algorithms, we use $[n] = \{1,\dots,n\}$.

Fix $k$, $n$, and $r$ and assume that all values of $\aw([m],k)$ have been computed for $k \leq m < n$.
Let $c : [n] \to [r] \cup \{*\}$ be a function called a \emph{partial $r$-coloring}, where every position $i$ has color $c(i) \in[r]$ or $c(i) = *$ and $i$ is \emph{uncolored}.
By starting with all positions uncolored, we recursively attempt to extend a partial $r$-coloring $c$ where the positions in $[i]$ are colored to an exact $r$-coloring $c'$ that avoids rainbow $k$-APs.
We branch at each recursive call for all possible choices of color for $c(i+1)$ such that no $k$-AP within $[i+1]$ is colored with $k$ distinct colors.
To guarantee that no chosen color creates a rainbow $k$-AP, we maintain a list of sets $D(j) \subseteq [r]$ that contain all of the possible colors for the position $j$.
Specifically, assigning $c(j)$ to be any color in $[r] \setminus D(j)$ will immediately create a rainbow $k$-AP.
Whenever a color is assigned to a position $i$, we consider a $k$-AP, $X$, whose second-to-last element is $i$.
If the set $c(X) = \{ c(i') : i' \in X - \max X \}$ contains $k-1$ distinct colors, we say that $X$ is an \emph{almost-rainbow $k$-AP} and the color for $\max X$ must be one of these $k-1$ colors.
Therefore, we can update $D(\max X)$ to be $D(\max X) \cap c(X)$.
For simplicity, we update $D(i)$ to be $\{ c(i) \}$ when $i$ is assigned the color $c(i)$.

We can also make a few small adjustments to greatly reduce the search space.
First, we assume that the coloring $c$ is lexicographically-minimum: for two colors $a, b \in [r]$ with $a < b$, we assume that the first position with color $a$ appears before the first position with color $b$.
Second, the domains $D(j)$ contain the possible colors for the positions that remain uncolored.
If $\bigcup_{j \in [n]} D(j) \neq [r]$, then $c$ cannot extend to an exact $r$-coloring.
Finally, if the first $i$ positions are all colored with the color $1$, then for any extension of $c$ to an exact $r$-coloring of $[n]$, the last $n - i + 1$ positions form an exact $r$-coloring.
Thus, if $\aw([n-i+1],k) \leq r$, then it is impossible to extend $c$ to an exact $r$-coloring of $[n]$ without creating a rainbow $k$-AP.

Our recursive algorithm is given as Algorithm~\ref{alg:norainbowkaps} and is initialized by Algorithm~\ref{alg:norainbowkapsbase}.
Similar algorithms are implemented for the case of $r$-coloring $\Z_n$.
All source code and data are available online\footnote{All source code and data can be found at 
%\url{http://www.math.iastate.edu/dstolee/r/rainbowaps.htm
\url{https://github.com/derrickstolee/RainbowAPs}} including computed values of $\aw([n],k)$ and $\aw(\Z_n,k)$, extremal colorings, and reports of computation time.

\begin{algorithm}
\caption{\label{alg:norainbowkaps}FindColorings($k, r, n, \aw, c, D, i$) -- Find exact $D$-colorings on $[n]$ that avoid rainbow $k$-APs and extend the coloring $c$ on $[i-1]$. Assume $\aw([m],k)$ is known for all $m < n$.}
\begin{algorithmic}
\IF{$i \equiv n$}
	\STATE \textbf{output} $c$
	\RETURN
\ELSIF{$\cup_{j\in [n]} D(j) \neq [r]$}
	\RETURN \textit{// This coloring cannot extend to an exact $r$-coloring!}
\ELSIF{$i > 2$ \textbf{and} $\forall j < i, c(j) \equiv 1$ \textbf{and} $\aw([n - i + 2], k) \leq r$}
	\RETURN \textit{// An exact $r$-coloring extending $c$ induces an exact $r$-coloring on $\{i-1,\dots,n\}$.}
\ENDIF
\STATE $M \leftarrow \max\{ c(j) : j < i\} \cup \{0\}$
\STATE \textit{// Attempt all colors in the domain $D(i)$ that are at most $M+1$.}
\FOR{\textbf{all} $a \in D(i) \cap [M+1]$}
	\STATE $c(i) \leftarrow a$, \quad $D(i) \leftarrow \{ a\}$
	\STATE \textit{// Update all domains $D'(t)$ when almost-rainbow $k$-APs exist.}
	% this is a crucial step to improving the performance of the algorithm!m
	\STATE $D' \leftarrow D$
	\FOR{\textbf{all} $d \in \{1,\dots,\lceil i/(k-2)\rceil-1\}$}
		\STATE $A \leftarrow \varnothing$
		\FOR{\textbf{all} $\ell \in \{0,\dots,k-2\}$}
			\STATE $t \leftarrow i - \ell\cdot d$
			\STATE $A \leftarrow A \cup \{ c(t)\}$
		\ENDFOR
		\IF{$|A| \equiv k-1$}
			\STATE $t \leftarrow i + d$
			\STATE $D'(t) \leftarrow D'(t) \cap A$
		\ENDIF
	\ENDFOR
	\STATE \textbf{call} FindColorings($k,r,n,\aw, c, D', i+1$)
\ENDFOR
\end{algorithmic}
\end{algorithm}

\begin{algorithm}
\caption{\label{alg:norainbowkapsbase}FindColoring($k, r, n, \aw$) -- Find exact $r$-colorings on $[n]$ that avoid rainbow $k$-APs.}
\begin{algorithmic}
\FOR{\textbf{all} $i \in [n]$}
	\STATE $c(i) \leftarrow *$
	\STATE $D(i) \leftarrow [r]$
\ENDFOR
\STATE \textbf{call} FindColorings($k, r, n, \aw, c, D, 1$)
\end{algorithmic}
\end{algorithm}

%\clearpage
%\section{Exact Values from Computation}

%Table \ref{tab:znallk}  \ref{tab:znallk} moves to Section \ref{ssect:awZnk}
%\clearpage

%%%%%%%%%%%%%%%%%%%%%%%%%%%%%%%%%%%%%%%%%%%%%%%%%%%%

 \section{Conjectures and open questions}\label{sec:conclude}
  We conclude by summarizing some open questions and conjectures, beginning with those related to $[n]$.

Uherka~\cite{U13} observed that $\aw([n],3)$ is not a monotone function in $n$, as there are values of $n$ where $\aw([n],3)=\aw([n-1],3)-1$. Does this happen infinitely often? Are larger drops possible?

\begin{conj}
For   positive integers $n$ and $k$, $\aw([n],k)\geq \aw([n-1],k)-1$.
\end{conj}

Conjecture~\ref{RMconj} states that the lower bound $\aw([n],3)\geq \lceil\log_3n\rceil+2$ is correct to within an additive constant. We further conjecture that the lower bound in Lemma~\ref{awn3log3lowerbnd} is in fact the exact value when $n$ is a power of three.  It is true for the computed data available (see Remark \ref{RK:obs:values}).

\begin{conj}\label{LHaw3tothem3}  Let $m$ be a nonnegative integer. Then
$\aw([3^m],3)=m+2$.
\end{conj}

\begin{quest}\label{LHaw3n3}
Is it true that $\aw([3n],3)=\aw([n],3)+1$ for all positive integers $n$?
\end{quest}

%Recall that a singleton extremal coloring of $S$ is an exact coloring of $S$ that avoids rainbow $3$-APs, uses exactly $\aw(S,3)-1$ colors, and has a color class of size 1. Singleton extremal colorings exist for all positive integers $n$ for $S=\Z_n$.  

%\begin{quest}\label{conj:singleton}  Does there exists a singleton extremal coloring of $[n]$ for every positive integer $n$? \end{quest}

We now turn our attention to $\Z_n$.  
\begin{quest}\label{quest:infprimes} Are there infinitely many primes $p$ such that $\aw(\Z_p,3)=3$? \end{quest}
Based on \cite[Theorem 3.5]{JLMNR03} (see also Theorem \ref{JLMNR3.5}), one approach to finding primes $p$ for which $\aw(\Z_p,3)=3$ is to search for primes $p$ such that the multiplicative group $\Z_p^\times$ is generated by $2$.  %In that case, a singleton extremal coloring of $\Z_p$ would have only two colors.
 However, the existence of an infinite family of such primes is still open.

\begin{conj}[Artin's Conjecture]\label{CH:quest:awznqfour} {\rm  \cite[p. 217]{S}}
	There are infinitely many primes $p$ such that $2$ is a generator of the multiplicative group $\Z_p^\times$.
\end{conj}

If Artin's Conjecture holds,  it would give us an infinite family of $\Z_p$ such that $\aw(\Z_p,3)=3$.  Jungi\'c et al. also established another family of primes $p$ with $\aw(\Z_p,3)=3$ (see Theorem \ref{JLMNR3.5}), namely those primes $p$ such that  $\frac{p-1}{2}$ is odd and the order of 2 in $\Z_p^\times$ is $\frac{p-1}{2}$.

\section*{Acknowledgements}

We thank the referee for a very helpful report, including pointing out some of the consequences \cite[Theorem 3.5]{JLMNR03} and suggesting the current proof of Lemma~\ref{LHprop220nolab}, which we had proved another way using Proposition~\ref{LHj2}.  The majority of this research was done during the 2013-2014 academic year at Iowa State University.

%%%%%%%%%%%%%%%%%%%%%%%%%%%%%%%%%%%%%%%%%%%%%

%%%%%%%%%%%%%%%%%%%%%%%%%%%%%%%%%%%%%%%%%%%%%%%%%%%%
%%%%%%%%%%%%%%%%%%%%%%%%%%%%%%%%%%%%%%%%%%%%%%%%%%%%

\end{document}